\documentclass[hidelinks,onefignum,onetabnum]{siamart250211}
\newcommand{\R}{\mathbb{R}}
\usepackage{tikz-cd}
\usepackage{tabularray}
\usetikzlibrary[math]
\usepackage{enumitem}

\usepackage{lipsum}
\usepackage{amsfonts}
\usepackage{graphicx}
\usepackage{epstopdf}
\usepackage{algorithmic}
\ifpdf
  \DeclareGraphicsExtensions{.eps,.pdf,.png,.jpg}
\else
  \DeclareGraphicsExtensions{.eps}
\fi

\newsiamremark{remark}{Remark}
\newsiamremark{hypothesis}{Hypothesis}
\crefname{hypothesis}{Hypothesis}{Hypotheses}
\newsiamthm{claim}{Claim}
\newsiamremark{fact}{Fact}
\crefname{fact}{Fact}{Facts}

\headers{Splitting Method for a Multilayered FPSI Problem}{ANDREW SCHARF, MARTINA BUKAČ, and SUNČICA ČANIĆ}

\title{Splitting Method for a Multilayered Poroelastic Solid Interacting with Stokes Flow\thanks{Submitted to the editors DATE.
\funding{The work of the first author was partially supported by the NSF through grant 
DMS-2011319.
The work of the second author was partially supported by the NSF through grants DMS-2208219 and DMS-2205695.
The work of the third author was partially supported by the NSF through grants
DMS-2408928, DMS-2247000 and DMS-2011319.}}}

\author{ ANDREW SCHARF \thanks{Department of Mathematics, University of California, Berkeley, Berkeley, CA 94720 
  (\email{scharfa@berkeley.edu, canics@berkeley.edu}).}
\and MARTINA BUKAČ \thanks{Department of Applied and Computational Mathematics and Statistics, University of Notre Dame, Notre Dame, IN 46556
  (\email{martina.bukac.1@nd.edu}).}
\and SUNČICA ČANIĆ \footnotemark[2]} 

\usepackage{amsopn}

\ifpdf
\hypersetup{
  pdftitle={Splitting Method for a Multilayered Poroelastic Solid Interacting with Stokes Flow},
  pdfauthor={ANDREW SCHARF, MARTINA BUKAČ, and SUNČICA ČANIĆ}
}
\fi

\begin{document}

\maketitle

% REQUIRED
\begin{abstract}
Multilayered poroelastic structures are found in many biological tissues such as cartilage and the cornea, and play a key role in the design of bioartificial organs and other bioengineering applications. Motivated by these applications, we study the interaction between a free fluid flow, governed by the time-dependent Stokes equations, and a multilayered poroelastic structure composed of a thick Biot layer and a thin, linear poroelastic plate located at the interface. The resulting equations are linearly coupled across the thin structure domain through physical coupling conditions. We develop a partitioned numerical scheme for this poroelastic fluid-structure interaction problem, combining the backward Euler Stokes-Biot splitting method with the fixed-strain Biot splitting approach. The first decouples the Stokes problem from the multilayered structure problem, while the second decouples the flow and mechanical subproblems within the poroelastic structures. Stability of the splitting scheme is proven under different combinations of time-step conditions and parameter constraints. The method is validated using manufactured solutions, and further applied to a biologically inspired blood vessel flow problem. We also demonstrate convergence of the solution to the limiting case without the plate as its thickness tends to zero, providing additional validation of the numerical method.
\end{abstract}

% REQUIRED
\begin{keywords}
poroelasticity, fluid-poroelastic structure interaction, poroelastic plate, multilayered poroelasticity, partitioned numerical methods, ﬁnite element method
\end{keywords}

% REQUIRED
\begin{MSCcodes}
74F10, 76S05, 74L15, 34A01, 74S05, 76M10, 65M60, 65M12, 65M22, 74H15
\end{MSCcodes}

\section{Introduction}
Poroelastic materials are deformable porous media that exhibit both fluid-dynamic and elastodynamic properties. The classic continuum mechanical model introduced by Biot \cite{Biot1} was originally motivated by problems in soil consolidation, and has since seen use in a wide array of geomechanics applications, see e.g. \cite{geoApp2,geoApp3,geoApp4,geoApp1}. Applications of poroelastic models have subsequently been found in a variety of fields; in particular, poroelasticity is frequently encountered in biological tissues such as cartilage, liver, cornea, arterial endothelium, and intervertebral disc \cite{bioApp2,bioApp3,bioApp1}. These structures are often multilayered and in contact with a free-flowing fluid, as in the case of blood passing through an arterial vessel. Similar multiphysics systems can occur in bioengineering designs, for example, that of an implantable bioartificial pancreas (iBAP) \cite{iBAP} consisting of a poroelastic islet cell scaffold encapsulated by a semipermeable silicon nanopore plate. This membrane acts as a selective barrier to isolate the transplanted cells from the host immune system, eliminating the need for immunosuppressive drugs. Developing efficient and accurate numerical schemes for simulating fluids interacting with multilayered poroelastic structures is therefore an important step in the computational analysis of biomedical devices and an area of active research (see \cite{iBAPNum2,iBAPNum1}).

To that end, in this paper we study a fluid-poroelastic structure interaction (FPSI) problem where the fluid is an incompressible viscous Newtonian fluid described by the linear Stokes equations and the structure consists of two Biot layers: a thick poroelastic material and a thin poroelastic plate with transverse flow. Each Biot system comprises two coupled equations: one governing the elastodynamics of the poroelastic medium, and the other expressing fluid mass conservation in terms of pore pressure. The fluid velocity can then be recovered using Darcy’s law. 
The fluid and thick Biot domains are linearly coupled across the poroelastic plate—which serves as an interface bearing mass and poroelastic energy—through a set of kinematic and dynamic coupling conditions.  The poroelastic plate model considered here was rigorously justified in \cite{BiotPlate}, and the FPSI problem studied in this manuscript was shown to have weak solutions that are unique under additional regularity assumptions in \cite{BociuMulti}. While other multilayered fluid-structure interaction (FSI) problems have previously been analyzed, e.g. in \cite{multiFSI2,MUHA2014,multiFSI1}, the use of a poroelastic rather than elastic membrane/plate interface presents unique challenges inherent to the Biot plate equations: the elastodynamics equation is a 4th-order equation capturing transverse (vertical) plate displacement, the fluid conservation of mass equation describes flow in only transverse direction, and the two equations are defined on domains of different dimension.

One well-studied linear FPSI model is the coupled Stokes-Biot (SB) system, which combines elements of FSI between a fluid and a deformable elastic body with elements of Stokes-Darcy (SD) coupling between a fluid and a rigid porous material. Numerical schemes for SB coupling (and coupling in general) can be monolithic, where the Stokes and Biot equations are solved simultaneously, or partitioned, where the subproblems are decoupled and solved independently from each other. Monolithic SB solvers have been developed and used for hemodynamics applications as in \cite{MONO4,MONO3,MONO2,MONO5,MONO6,MONO1}. These schemes have the advantage of treating coupling conditions implicitly; however, since all variables are solved together, they often result in large systems that rely on preconditioning or other modifications in order to achieve stability. In contrast, partitioned methods benefit from modularity (i.e. using existing methods and code for Stokes and Biot subproblems) and better conditioning due to the resulting smaller systems in each subproblem. On the other hand, stability for splitting schemes is highly sensitive to how the coupling terms are handled numerically at the interface. Strongly-coupled partitioned methods, which perform subiterations of solving the subproblems to converge to a solution at each time-step, achieve better enforcement of coupling conditions at the expense of more computation. Loosely-coupled methods, which solve each subproblem only once in a single time-step, are computationally less expensive but may suffer reduced stability or accuracy compared to strongly-coupled alternatives. A number of partitioned schemes for FPSI have been developed, such as in \cite{Mortar,NitCoup,MultiLayer,FPSI2,RobRob,MACscheme}, and frequently borrow tools and splitting methods used for SD and FSI problems. In \cite{MultiLayer}, a loosely-coupled splitting method is used to solve a fluid multilayered-structure interaction problem where the structure consists of a thin {\emph{elastic}} layer and a thick poroelastic layer. While the structure model in that paper was purely elastic, it was assumed that the membrane is permeable to fluids. Building on the kinematically coupled scheme -- an unconditionally stable, loosely coupled method that has been successfully applied to various FSI problems \cite{KinCoup} -- the authors in \cite{MultiLayer} demonstrate that the scheme in \cite{MultiLayer} achieves conditional stability under a CFL condition. Since fluid multilayered-structure interaction involves an even greater number of variables than the standard SB FPSI problem, loosely-coupled partitioned methods potentially offer even greater advantages in reducing numerical costs while requiring a more complex consideration of how to effectively treat coupling conditions at the interface.

The poroelastic plate in our model introduces additional numerical challenges due to the unidirectional fluid flow—restricted to the transverse (normal) direction—and the presence of a fourth-order spatial derivative. Additionally, while the plate elastodynamics subproblem occurs on the codimension 1 interface -- middle surface of the plate -- over which the Stokes and Biot domains are coupled, the hydrodynamics subproblem is defined on the codimension 0 domain of small but finite thickness $H$. Previous works, see \cite{PoroPlateNum1,PoroPlateNum2},  have developed numerical schemes that circumvent this mismatch in domains by averaging the fluid subproblem over the transverse dimension, resulting in a system of equations on the middle surface of the plate. This strategy is less naturally suited for a multiphysics problem where the plate is coupled to other fluid domains, since continuity of pressures and normal velocities are coupling conditions that are imposed in a trace sense and not in an average sense. To our knowledge, prior to this work no numerical scheme has been developed that solves the Biot poroelastic plate equations coupled to any other model. 

In this paper, we introduce and analyze a novel loosely-coupled splitting scheme for solving the Stokes--Biot plate--Biot coupled problem described in \cref{Model,Energy}. Our strategy outlined in \cref{Scheme} combines a Stokes-Biot splitting of the free flowing fluid in the Stokes domain from the multilayered poroelastic structure, followed by an additional Biot-splitting method to decouple the hydrodynamic and elastodynamic subproblems in the Biot domains. In particular, we make use of fixed-strain Biot splitting \cite{BiotSplit,FixedStrain}, which gives rise to a loosely-coupled partitioned scheme that, for the thick Biot domain, is conditionally stable under a time-step condition and unconditionally stable under a constraint on the problem parameters. We extend this method to the Biot plate equations with the ``anisotropic'' flow restricted to only transverse direction, and show analagous stability properties as in the full Biot case with unrestricted flow. The three resulting subproblems consist of a Darcy solve, an elastic solid solve, and a Stokes solve, each of which are simple to implement and can readily make use of already existing code. Discretization in space is carried out via the finite element method, where mixed finite elements are used for the Stokes subproblem and the elastic plate subproblem. We analyze the stability of this scheme in \cref{Stability} and show that we have conditional stability with constraints mirroring those arising from the individual Biot and SB splittings, which under certain parameter constraints is a CFL condition. We then demonstrate in \cref{Simulations} the efficacy of the method in a number of simulations. Using the method of manufactured solutions, we show that the scheme exhibits first-order convergence in time and second-order convergence in space. 
We also validate the solver by demonstrating that solutions to the Stokes–Biot plate–Biot problem, with plate thickness $H$, converge as $H\to0$ to the solution of the corresponding Stokes–Biot problem obtained via a separate, monolithic approach. Convergence plots at the interface indicate that the presence of a thin plate carrying mass and poroelastic energy acts to “regularize” the solution—smoothing or damping the coupled FPSI dynamics.
Finally, we apply the scheme to a biologically motivated hemodynamics FPSI problem, and show that the method performs robustly under physiologically relevant parameter regimes.

\section{Mathematical Model} \label{Model}

\subsection{Spatial Domain $\Omega$} \label{Spatial}

We study the interaction of Stokes flow of an incompressible, Newtonian fluid with a multilayered poroelastic structure in three physical space dimensions. The structure consists of a thick layer $\Omega_b$ and a thin plate $\Omega_p$ with mass, both of which are porous and deformable. The plate acts as the interface between $\Omega_b$ and the fluid domain $\Omega_f$. We imagine the entire domain $\Omega$ to be a rectangular prism $I_1\times I_2\times I_3 \subset \R^3$ where $0$ is the midpoint of the interval $I_3$; thus the intersection of the plane $z = 0$ with $\Omega$, represented by the 2D rectangle $\Gamma$ in the $xy$-plane, separates the domain into an upper (thick) structure domain $\Omega_b := \Omega\cap\{z \geq 0\}$ and a lower fluid domain $\Omega_f := \Omega\cap\{z \leq 0\}$ (see \cref{fig:PhysDomain}). The thin plate domain, which we specify to have a (small) thickness $H$, is given by $\Omega_p := \Gamma\times\left[-\frac{H}{2},\frac{H}{2}\right]$ and has upper and lower boundaries $\Gamma_\pm := \Gamma\times\left\{\pm\frac{H}{2}\right\}$. With a slight abuse of notation we write $\Gamma$ for the middle surface $\Gamma\times\{0\}$ of the plate bounding $\Omega_f$ and $\Omega_b$. We write $\boldsymbol n^p = (0,0,1)$ for the unit normal vector of the membrane so that $\nabla_{\boldsymbol n^p} = \frac{\partial}{\partial z}$ is the normal derivative to $\Gamma$, and similarly write $\boldsymbol \tau^p_i, i = 1, 2$ to denote directions tangent to the membrane so that $\nabla_{\boldsymbol \tau^p} = \left(\frac{\partial}{\partial x},\frac{\partial}{\partial y}\right)$ is the 2D tangential gradient operator.

\tikzmath{\XD = 6; \YD =.8; \ZD = 3;\PD=.2;}
\tikzmath{\RR = 1.5; \LL =1.5; \HH = .21;}  
\usetikzlibrary{shapes.geometric}
\usetikzlibrary{calc}
\usetikzlibrary{arrows.meta,bending,positioning}
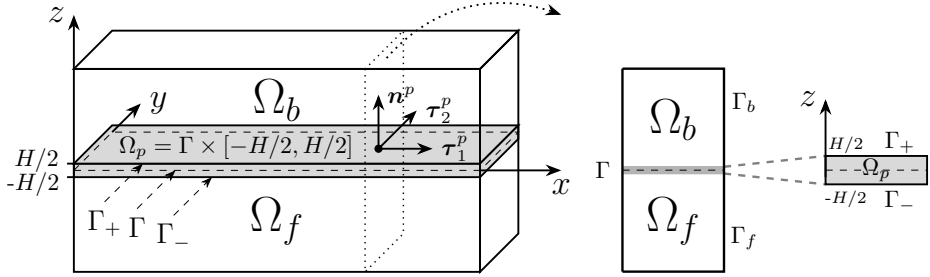
\begin{figure}[h]
\centering
\scalebox{.9}{\begin{tikzpicture}
\draw[fill=gray!30, thick] (0,-\PD/2) rectangle (\XD,\PD/2);
\draw[fill=gray!30, thick] (0, \PD/2) -- (\XD, \PD/2) -- ($(\XD, \PD/2) + (45:\YD)$) -- ($(0, \PD/2) + (45:\YD)$) -- cycle node[shift={(4*\XD/10,3*\YD/10)}] {$\Omega_p = \Gamma\times[-H/2,H/2]$} ;
\draw[fill=gray!30, thick] (\XD, \PD/2) -- (\XD, -\PD/2) -- ($(\XD, -\PD/2) + (45:\YD)$) -- ($(\XD, \PD/2) + (45:\YD)$) -- cycle;
\draw[line width=.6,dotted] (.717*\XD, \ZD/2) -- (.717*\XD, -\ZD/2) -- ($(.717*\XD, -\ZD/2) + (45:\YD)$) -- ($(.717*\XD, \ZD/2) + (45:\YD)$) -- cycle;

\draw[dashed] (0, 0) -- (\XD, 0) -- ($(\XD, 0) + (45:\YD)$) -- ($(0, 0) + (45:\YD)$) -- cycle;

\draw[dashed,-Stealth]  ($(.34*\XD,-\PD/2) - (45:\YD)$) node[below] {\large $\Gamma_-$} -- (.34*\XD,-\PD/2) ;
\draw[dashed,-Stealth]  ($(\XD/4,0) - (45:\YD)$) node[below] {\large $\Gamma$} -- (\XD/4,0) ;
\draw[dashed,-Stealth]  ($(.17*\XD,\PD/2) - (45:\YD)$) node[below] {\large $\Gamma_+$} -- (.17*\XD,\PD/2) ;

\draw[fill=black,fill opacity=1] (.75*\XD,4*\YD/10) circle (1.5pt);
\draw[thick,-Stealth]  (.75*\XD,4*\YD/10) -- (.75*\XD,4*\YD/10+\YD) node[right] { $\boldsymbol n^p$};
\draw[thick,-Stealth]  (.75*\XD,4*\YD/10) -- (.75*\XD + \YD,4*\YD/10) node[right] { $\boldsymbol{\tau}^p_1$};
\draw[thick,-Stealth]  (.75*\XD, 4*\YD/10) -- ($(.75*\XD, 4*\YD/10) + (45:\YD)$) node[right] {$\boldsymbol\tau^p_2$};

\draw[thick,-Stealth]  (0,\ZD/2) -- (0,\ZD/2+\YD) node[left] {\Large $z$};
\draw[thick,-Stealth]  (\XD,0) -- (\XD + 3*\YD/2,0) node[below] {\Large $x$};
\draw[thick,-Stealth]  ($(0, 0) + (45:\YD)$) -- ($(0, 0) + (45:7*\YD/4)$) node[right] {\Large $y$};
\draw[thick] (-.1,\PD/2) node[shift={(-.45,.06)}]  {\color{black} $H/2$} -- (0,\PD/2) ;
\draw[thick] (-.1,-\PD/2)  node[shift={(-.5,-.1)}]  {\color{black} -$H/2$} -- (0,-\PD/2) ;

\draw[thick] (0,-\ZD/2) rectangle (\XD,\ZD/2) node[shift={(-\XD/2,-\ZD/6.5)}] {\huge\color{black}$\Omega_b$} node[shift={(-\XD/2,-3*\ZD/4)}] {\huge\color{black}$\Omega_f$};
\draw[thick] (0, \ZD/2) -- (\XD, \ZD/2) -- ($(\XD, \ZD/2) + (45:\YD)$) -- ($(0, \ZD/2) + (45:\YD)$) -- cycle;
\draw[thick] (\XD, \ZD/2) -- (\XD, -\ZD/2) -- ($(\XD, -\ZD/2) + (45:\YD)$) -- ($(\XD, \ZD/2) + (45:\YD)$) -- cycle;

\draw[thick,-Stealth,dotted] (.77*\XD,\ZD/2 + \YD/5) to[bend left=35] (\XD + 3*\YD/2,\ZD/2+\YD);
\end{tikzpicture}}
\scalebox{.9}{\begin{tikzpicture}[text opacity=1,line width= 1pt]
\draw (0,-\RR) -- (0,0);
\draw (0,0) -- (0,\RR);
\draw (0,\RR) -- (\LL,\RR) node[shift={(.3,-\RR/3)}] { $\Gamma_b$} -- (\LL,-\RR) node[shift={(.3,\RR/3)}] { $\Gamma_f$} -- (0,-\RR);
\node at (\LL/2,-\RR/2) {\huge $\Omega_f$};
\node at (\LL/2,\RR/2) {\huge $\Omega_b$};
\draw[opacity=.3,line width=3.5pt] (0,0) node[left] { $\Gamma$} -- (\LL,0);

\draw[fill=gray!30, thick] (3*\LL,-\HH) node[shift={(-.4,-.25)}] { $\Gamma_-$} -- (3*\LL,\HH) node[shift={(-.4,.2)}] { $\Gamma_+$} -- (2*\LL,\HH) node[shift={(.3,.17)}] { \scriptsize$H/2$} -- (2*\LL,-\HH) node[shift={(.3,-.22)}] { \scriptsize -$H/2$} -- cycle;

\draw[dashed,line width=.3] (2*\LL, 0) -- (3*\LL, 0);
\draw[dashed,line width=.3] (0, 0) -- (\LL, 0);
\node at (2.5*\LL,0) { $\Omega_p$};
\draw[line width=1pt,dashed,opacity = .5] (\LL,.05)  -- (2*\LL,\HH) ;
\draw[line width=1pt,dashed, opacity = .5] (\LL,-.05)  -- (2*\LL,-\HH);
\draw[thick,-Stealth]  (2*\LL,-\HH/2) -- (2*\LL,7*\RR/10) node[left] {\Large $z$};
\end{tikzpicture}}

\caption{3D domain $\Omega$ consisting of a poroelastic medium $\Omega_b$ and a fluid domain $\Omega_f$ coupled across a poroelastic plate $\Omega_p$ (left), 2D vertical cross-section with inflated poroelastic plate $\Omega_p$ (right)}
\label{fig:PhysDomain}
\end{figure}

For any function $f$ defined on $\Gamma$, we can define its {\bf{extension}} $\tilde f$ on $\Omega_p$ by taking 
\begin{equation}\label{extension}
\left.\tilde f\right|_\Gamma = f\ \text{with}\  \nabla_{\boldsymbol n^p}\tilde f = 0 \ \text{on} \ \Omega_p; \ \text{i.e.} \ \tilde f(x,y,z) = f(x,y).
\end{equation}
Conversely, for any function $f$ on $\Omega_p$, we can consider the {\bf{averaged function}} $\overline{f} = \frac{1}{H}\int_{-H/2}^{H/2}f \ dz$ defined on $\Gamma$. The boundaries of $\Omega_b$ and $\Omega_f$ external to $\Omega$ are denoted $\Gamma_b$ and $ \Gamma_f$, respectively, so that $\partial\Omega_b = \Gamma_b\cup\Gamma$ and $\partial\Omega_f = \Gamma_f\cup\Gamma$. The $L^2$ inner product on a domain $\mathcal O$ will be written as $\left(\cdot,\cdot\right)_\mathcal O$ with corresponding norm $||\cdot||_\mathcal O$.

\subsection{Biot Model in $\Omega_b$} \label{Biot}

In $\Omega_b$, the unknowns are the displacement $\boldsymbol \eta$, the fluid pore pressure $p$, and the relative filtration velocity $\boldsymbol u_b$, which satisfy the Biot equations \cite{Biot1} derived from the balance of linear momentum, conservation of mass, and Darcy's law, respectively:
\begin{subequations}\label{eq:thick}\begin{align} & \rho_{b} \partial_{t}^2 \boldsymbol{\eta}-\nabla \cdot \boldsymbol{\sigma}_{b}(\boldsymbol \eta,p) + \gamma \boldsymbol\eta= \boldsymbol{F}_b,   \label{eq:thick1} \\
& \partial_{t}(c_{0}  p+\alpha \nabla \cdot  \boldsymbol{\eta}) + \nabla \cdot \boldsymbol u_b = G_b, \label{eq:thick2} \\
& \boldsymbol\kappa^{-1}\boldsymbol{u}_b = -\nabla p. \label{eq:thick3}\end{align}\end{subequations}
Here $\boldsymbol \sigma_b(\boldsymbol \eta,p)$ and $\boldsymbol \kappa$ are the poroelastic stress and permeability tensors, $\boldsymbol F_b$ and $G_b$ are forcing terms, $\rho_b$ is the structure density, $c_0$ is the storage coefficient, and $\alpha$ is the Biot-Willis parameter representing the strength of fluid-solid coupling. In the 3D case, the coefficient $\gamma$ is zero; however, in the 2D axially symmetric formulation (e.g. found in blood vessel models), $\gamma > 0$ acts as a spring term in order to couple the top and bottom stucture displacements (see for example \cite{MultiLayer,SprCoef1,SprCoef2,SprCoef3}). We further specify $\boldsymbol \sigma_b(\boldsymbol\eta,p) = \boldsymbol\sigma_e(\boldsymbol\eta) - \alpha p\boldsymbol I$ where $\boldsymbol\sigma_e(\boldsymbol\eta) = 2\mu_b\boldsymbol D(\boldsymbol\eta) + \lambda_b (\nabla\cdot\boldsymbol\eta)\boldsymbol I $ is the elastic stress tensor for a Saint Venant-Kichhoff material with Lam\'{e} parameters $\mu_b, \lambda_b$ and $\boldsymbol D(\boldsymbol\eta) = \frac{1}{2}(\nabla\boldsymbol\eta + (\nabla\boldsymbol\eta)^T)$ represents the symmetrized gradient. In addition, we assume $\boldsymbol \kappa$ is a positive definite tensor with minimum eigenvalue $k_{min}$. Letting $\boldsymbol \xi = \partial_t\boldsymbol\eta$ denote the structure velocity of the system, the full set of first-order in time equations in $\Omega_b$ is given by
\begin{align*}
\left\{\begin{array}{l}
\partial_{t} \boldsymbol{\eta} = \boldsymbol{\xi}, \\
\rho_{b} \partial_{t} \boldsymbol{\xi}-2 \mu_{b}\nabla \cdot  \boldsymbol{D}(\boldsymbol{\eta})-\lambda_{b}\nabla(\nabla \cdot \boldsymbol{\eta}) + \alpha \nabla p +  \gamma \boldsymbol\eta= \boldsymbol{F}_b, \\
c_{0} \partial_{t} p+\alpha \nabla \cdot  \boldsymbol{\xi} + \nabla \cdot \boldsymbol u_b = G_b, \\
\boldsymbol\kappa^{-1}\boldsymbol{u}_b = -\nabla p.
\end{array}\right. 
\end{align*}

\subsection{Poroelastic Plate Model in $\Omega_p$} \label{Plate}
 For a poroelastic plate it was shown in \cite{BiotPlate} that the dominant flow through the plate is in transverse direction. The unknowns are the transverse displacement of the plate $w$ defined on the middle surface $\Gamma$, the fluid pore pressure $q$, and the transverse filtration velocity $u_p$. These satisfy the Biot poroelastic plate equations (derived by Biot \cite{Biot2,Biot3} and rigorously justified by Marciniak-Czochra and Mikeli\'{c} \cite{BiotPlate}) given by
\begin{subequations}\label{eq:thin}\begin{align} 
    &H\rho_{p} \partial_{t}^2  w+H^3D \Delta_{\boldsymbol{\tau}^{p}}^2  w +H\gamma^p w+H \alpha^{p} \Delta_{\boldsymbol{\tau}^{p}} \overline{z q}=F_{p}   &&\text{ in } \Gamma\times(0,T), \label{eq:thin1}\\
&\partial_{t}(c_{0}^{p}  q-\alpha^{p} z \Delta_{\boldsymbol{\tau}^{p}}\tilde w)+\nabla_{\boldsymbol{n}^{p}}u_p = G_p &&\text{ in } \Omega_p\times(0,T), \label{eq:thin2}\\
& (\kappa^p)^{-1}u_p = -\nabla_{\boldsymbol{n}^{p}} q &&\text{ in } \Omega_p\times(0,T) . \label{eq:thin3}
\end{align}\end{subequations}
Here $\kappa^p > 0$ is a permeability coefficient, $F_p$ and $G_p$ are forcing terms, $\rho_p$ is the plate area density, $c_0^p$ is the storage coefficient, $\alpha^p$ is the Biot-Willis parameter representing the strength of fluid-solid coupling, $D$ is the bending stiffness of the plate, and $\gamma^p$ is a 2D spring coefficient that plays a similar role to that of $\gamma$ in the thick domain. As in the Biot model, we set $v = \partial_t w$ and $\Lambda = -\Delta_{\boldsymbol{\tau}^{p}} w$ to obtain the reduced-order system
\begin{align*}
\left\{\begin{array}{ll}
\partial_t w = v & \text{ in } \Gamma\times(0,T), \\
H\rho_{p} \partial_{t} v-H^3D \Delta_{\boldsymbol{\tau}^{p}} \Lambda+H\gamma^p w+H \alpha^{p} \Delta_{\boldsymbol{\tau}^{p}} \overline{z q}=F_{p} & \text{ in } \Gamma\times(0,T), \\
\partial_t\Lambda = -\Delta_{\boldsymbol{\tau}^{p}} v & \text{ in } \Gamma\times(0,T), \\
c_{0}^{p} \partial_{t} q-\alpha^{p} z \Delta_{\boldsymbol{\tau}^{p}}\tilde v+\nabla_{\boldsymbol{n}^{p}}u_p = G_p & \text{ in } \Omega_p\times(0,T), \\
(\kappa^p)^{-1}u_p = -\nabla_{\boldsymbol{n}^{p}} q & \text{ in } \Omega_p\times(0,T).
\end{array}\right.  
\end{align*}

\subsection{Stokes Model in $\Omega_f$} \label{Stokes}

In the fluid domain $\Omega_f$, the fluid velocity $\boldsymbol u$ and pressure $\pi$ satisfy the time-dependent incompressible Stokes equations
\begin{subequations}\label{eq:fluid}\begin{align} 
&  \rho_{f} \partial_{t} \boldsymbol{u}-\nabla \cdot \boldsymbol{\sigma}_{f}(\boldsymbol u, \pi) = \boldsymbol F_f, \label{eq:fluid1} \\
&  \nabla\cdot\boldsymbol u = 0,\label{eq:fluid2}
\end{align}\end{subequations}
with fluid stress tensor $\boldsymbol\sigma_f(\boldsymbol u, p_f) = -\pi\boldsymbol I + 2\mu_f\boldsymbol D(\boldsymbol u)$, density $\rho_f$, and viscosity $\mu_f$.

\subsection{Initial, Boundary, and Coupling Conditions} \label{Conditions}
For initial conditions, we take $\boldsymbol\eta_0, \boldsymbol\xi_0,p_0,w_0,v_0,q_0,\boldsymbol u_0$ to be the values at time $t = 0$.
On the domain boundary $\partial\Omega$, for ease of analysis we impose:
\begin{enumerate}[label=\textbf{(B\arabic*)}]
\item the clamped and drained boundary conditions $\boldsymbol\eta \equiv p \equiv 0$ on $\Gamma_b$;
\item the clamped boundary condition $w \equiv\frac{\partial }{\partial \boldsymbol\nu}w\equiv 0$ on $\partial\Gamma$ where $\frac{\partial }{\partial \boldsymbol\nu}$ is the normal derivative on $\partial\Gamma$;
\item the no-slip boundary condition $\boldsymbol u \equiv 0$ on $\Gamma_f$.
\end{enumerate}
Alternative boundary conditions are explored in \cref{Simulations} for numerical simulations. The problems are coupled along the interfaces $\Gamma, \Gamma_-, \Gamma_+$ by the following kinematic and dynamic coupling conditions:
\begin{enumerate}[label=\textbf{(BJS)}]
\item Beavers-Joseph-Saffman slip condition $\left.\boldsymbol\sigma_f(\boldsymbol u,\pi)\boldsymbol n^p\right|_\Gamma \cdot \boldsymbol\tau^p_i = -\left.\beta\boldsymbol u\right|_\Gamma\cdot \boldsymbol\tau^p_i$;
\end{enumerate}
\begin{enumerate}[label=\textbf{(K\arabic*)}]
\item continuity of normal components of velocity $\left.\boldsymbol u\right|_{\Gamma} \cdot\boldsymbol n^p - v= \left.u_p\right|_{\Gamma_-}$ and \\
$\left.\boldsymbol u_b\right|_{\Gamma} \cdot\boldsymbol n^p= \left.u_p\right|_{\Gamma_+}$;
\item continuity of displacement $ \left.\boldsymbol\eta\right|_\Gamma =  w\boldsymbol n^p$;
\end{enumerate}
\begin{enumerate}[label=\textbf{(D\arabic*)}]
\item balance of forces $F^p = \left.\boldsymbol\sigma_p(\boldsymbol \eta,p)\boldsymbol n^p\right|_\Gamma \cdot \boldsymbol n^p - \left.\boldsymbol\sigma_f(\boldsymbol u,\pi)\boldsymbol n^p\right|_\Gamma \cdot \boldsymbol n^p$;
\item and the continuity of pressure $ \left.q\right|_{\Gamma_-} = -\left.\boldsymbol\sigma_f(\boldsymbol u,\pi)\boldsymbol n^p\right|_\Gamma \cdot \boldsymbol n^p$ and $ \left.q\right|_{\Gamma_+} =  \left.p\right|_\Gamma $.
\end{enumerate}

The full set of coupling conditions is thus
\begin{align}\label{CC}
\left\{\begin{array}{lr}
 \left.q\right|_{\Gamma_+} =  \left.p\right|_\Gamma & \text{ in } \Gamma_+\times(0,T), \\
\left.\boldsymbol u_b\right|_{\Gamma} \cdot\boldsymbol n^p= \left.u_p\right|_{\Gamma_+} & \text{ in } \Gamma_+\times(0,T), \\
\left.\boldsymbol\eta\right|_\Gamma =  w \boldsymbol n^p & \text{ in } \Gamma\times(0,T), \\
F^p = \left.\boldsymbol\sigma_b(\boldsymbol \eta,p)\boldsymbol n^p\right|_\Gamma \cdot \boldsymbol n^p - \left.\boldsymbol\sigma_f(\boldsymbol u,\pi)\boldsymbol n^p\right|_\Gamma \cdot \boldsymbol n^p & \text{ in } \Gamma\times(0,T), \\
\left.\boldsymbol\sigma_f(\boldsymbol u,\pi)\boldsymbol n^p\right|_\Gamma \cdot \boldsymbol\tau^p_i = -\left.\beta\boldsymbol u\right|_\Gamma\cdot \boldsymbol\tau^p_i & \text{ in } \Gamma\times(0,T), \\
\left.q\right|_{\Gamma_-} = -\left.\boldsymbol\sigma_f(\boldsymbol u,\pi)\boldsymbol n^p\right|_\Gamma \cdot \boldsymbol n^p & \text{ in } \Gamma_-\times(0,T), \\
\left.u_p\right|_{\Gamma_-} = \left.\boldsymbol u\right|_{\Gamma} \cdot\boldsymbol n^p - v & \text{ in } \Gamma_-\times(0,T). \\
\end{array}\right.  
\end{align}

\section{Energy, Function Spaces, Weak Formulation} \label{Energy}

For the following we define the solution vectors $\boldsymbol A = (\boldsymbol\xi,v,\Lambda,p,\boldsymbol u_b,q,u_p,\boldsymbol u, \pi)$ and $\boldsymbol{\hat{A}} = \boldsymbol A \times (\boldsymbol\eta,w)$. We also consider a vector of test functions $\boldsymbol B = (\boldsymbol \phi, \varphi, \psi, r, \boldsymbol U_b, s, U_p, \boldsymbol U, \Pi)$.

By testing equations (1-3) with the solution variables and using the boundary and coupling conditions, we obtain the formal energy equality 
\begin{equation*}  \frac{1}{2}\dot{ \mathcal E}(\boldsymbol{\hat{A}}) + \mathcal D(\boldsymbol{\hat{A}}) = \mathcal F(\boldsymbol{\hat{A}}) \end{equation*} 
where
\begin{equation*} \begin{split}  \mathcal E(\boldsymbol{\hat{A}}) 
 =& \rho_b||\boldsymbol\xi||^2_{\Omega_b} + \lambda_b||\nabla\cdot\boldsymbol\eta||^2_{\Omega_b} + 2\mu_b||\boldsymbol D(\boldsymbol \eta)||^2_{\Omega_b} +  \gamma||\boldsymbol\eta||^2_{\Omega_b}  + H\rho_p||v||^2_{\Gamma} \\
& + H^3D||\Lambda||^2_{\Gamma} + H\gamma^p||w||^2_{\Gamma} +  c_0||p||^2_{\Omega_b} + c_0^p||q||^2_{\Omega_p} + \rho_f||\boldsymbol u||^2_{\Omega_f}, \\
\mathcal D(\boldsymbol{\hat{A}})  = & ||\boldsymbol\kappa^{-1/2}\boldsymbol u_b||^2_{\Omega_b} + ||(\kappa^p)^{-1/2} u_p||^2_{\Omega_p} + 2\mu_f||\boldsymbol D (\boldsymbol u)||^2_{\Omega_f} + \beta\left|\left|\boldsymbol P_{\boldsymbol\tau^p}\boldsymbol{u}\right|\right|^2_{\Gamma},  \\
\mathcal F(\boldsymbol{\hat{A}})  = &  (\boldsymbol F_b,\boldsymbol\xi)_{\Omega_b}  +  (G_b,p)_{\Omega_b} +  (G_p,q)_{\Omega_p} + (\boldsymbol F_f,\boldsymbol u)_{\Omega_f},
\end{split}  \end{equation*}
and $\boldsymbol P_{\boldsymbol\tau^p}\boldsymbol{u} = \sum_{i=1}^2\left(\left.\boldsymbol u\right|_\Gamma\cdot \boldsymbol\tau^p_i\right)\boldsymbol\tau^p_i$ denotes the projection of $\boldsymbol u$ onto the tangent space of $\Gamma$. A full derivation can be found in \cite{BociuMulti}.

For any function space $X$ on a domain $\mathcal O\subset\Omega$, let $X_{0} = \left\{f \in X : \left.f\right|_{\partial\Omega} \equiv 0\right\}$ be the subset of functions vanishing on $\partial\Omega$. Based on the formal energy estimate, we define the spaces
\begin{equation*} \begin{split} 
Q^s & = \left\{(p, \boldsymbol u_b, q, u_p) \in H^1_{0}(\Omega_b)\times \boldsymbol L^2(\Omega_b)\times H^1(\Omega_p)\times L^2(\Omega_p) : \left.q\right|_{\Gamma_{+}}=\left.p\right|_{\Gamma} \right\}, \\
 \mathcal V^b & = \left\{(\boldsymbol \xi,v) \in \boldsymbol H^1_{0}(\Omega_b)\times  H^1_0(\Gamma):\left.\boldsymbol\xi\right|_\Gamma =  v\boldsymbol n^p \right\},  \\
\mathcal V^s & = \mathcal V^b\times H^1(\Gamma), \\
 \mathcal V^f &=\boldsymbol H^1_{0}(\Omega_f)\times L^2(\Omega_f), \\
 \mathcal V & = Q^s \times \mathcal V^s\times \mathcal V^f.
\end{split} \end{equation*}

\begin{remark}
From the energy estimate, we only have estimates for $\nabla_{\boldsymbol n^p}q$ and not $\nabla_{\boldsymbol \tau^p}q$, so $q$ actually belongs to $H^{0,0,1}(\Omega_p) = \left\{q \in L^2(\Omega_p) : \nabla_{\boldsymbol n^p}q \in  L^2(\Omega_p)\right\}$. To avoid using anisotropic finite element spaces, which require specialized constructions, for the numerical scheme we work with $q \in H^{1}(\Omega_p)$ instead.
\end{remark}

The primal mixed weak formulation of the monolithic Stokes--Biot plate--Biot coupled problem is as follows: given $t \in [0,T]$, find solutions $\boldsymbol A \in \mathcal V, (\boldsymbol\eta,w) \in \mathcal V_b$ with $\partial_t\boldsymbol\eta = \boldsymbol \xi, \partial_t w = v$ so that for every $\boldsymbol B \in \mathcal V$:
\begin{equation*} \begin{split} 
&  \rho_b(\partial_t \boldsymbol\xi,\boldsymbol\phi)_{\Omega_b}  + \lambda_b(\nabla\cdot\boldsymbol\eta,\nabla\cdot\boldsymbol\phi)_{\Omega_b} + 2\mu_b(\boldsymbol D(\boldsymbol\eta),\boldsymbol D(\boldsymbol\phi))_{\Omega_b} - \alpha (p,\nabla\cdot\boldsymbol\phi)_{\Omega_b} + \gamma (\boldsymbol\eta,\boldsymbol\phi)_{\Omega_b} \\
&  + H\rho_p(\partial_t v,\varphi)_{\Gamma} + H^3D(\nabla_{\boldsymbol\tau^p}\Lambda,\nabla_{\boldsymbol\tau^p}\varphi )_\Gamma + H^3D(\partial_t\Lambda,\psi)_\Gamma - H^3D(\nabla_{\boldsymbol\tau^p}v,\nabla_{\boldsymbol\tau^p}\psi )_\Gamma   \\
& + H\gamma^p(w,\varphi)_{\Gamma} - H\alpha^p\left( \nabla_{\boldsymbol\tau^p}\overline{zq},\nabla_{\boldsymbol\tau^p}\varphi \right)_\Gamma   -   \left(q,\varphi  \right)_{\Gamma_-}   + c_0 (\partial_t p,r)_{\Omega_b} + \alpha(\nabla\cdot\boldsymbol\xi,r)_{\Omega_b}  \\
& - ( \boldsymbol u_b,\nabla r)_{\Omega_b} +  ( \nabla p,\boldsymbol U_b)_{\Omega_b}  +  ( \kappa^{-1}\boldsymbol u_b,\boldsymbol U_b)_{\Omega_b} + c_0^p (\partial_t q,s)_{\Omega_p} + \alpha^p(\nabla_{\boldsymbol\tau^p}\tilde v,\nabla_{\boldsymbol\tau^p}(zs))_{\Omega_p}   \\
&- (  u_p,\nabla_{\boldsymbol n^p} s)_{\Omega_p} +  ( \nabla_{\boldsymbol n^p} q, U_p)_{\Omega_p}  +  ( (\kappa^p)^{-1} u_p, U_p)_{\Omega_p} - \left(\left.\boldsymbol{u}\right|_{\Gamma} \cdot \boldsymbol{n}^{p}-v,s  \right)_{\Gamma_-}\\
&  +  \rho_f(\partial_t\boldsymbol u,\boldsymbol U)_{\Omega_f}   + 2\mu_f(\boldsymbol D(\boldsymbol u),\boldsymbol D(\boldsymbol U))_{\Omega_f} - (\pi,\nabla\cdot\boldsymbol U)_{\Omega_f} +  (\nabla\cdot\boldsymbol u,\Pi)_{\Omega_f} \\
& + \left(q,\left.\boldsymbol{U}\right|_\Gamma \cdot \boldsymbol{n}^{p} \right)_{\Gamma_-} + \beta\left(\boldsymbol P_{\boldsymbol{\tau}^p} \boldsymbol{u},\boldsymbol P_{\boldsymbol{\tau}^p} \boldsymbol{U}\right)_{\Gamma}  = (\boldsymbol F_b,\boldsymbol\phi)_{\Omega_b}  +  (G_b,r)_{\Omega_b} +  (G_p,s)_{\Omega_p} \\
& + (\boldsymbol F_f,\boldsymbol U)_{\Omega_f}  - \frac{\gamma_{pen}\mu_f}{d_h}\left(\left.\boldsymbol u\right|_\Gamma\cdot\boldsymbol n^p - v - u_p,\left.\boldsymbol U\right|_\Gamma\cdot\boldsymbol n^p - \varphi -  U_p\right)_{\Gamma_-} .
\end{split}  \end{equation*}

An extra term has been added at the end to more strictly enforce continuity of normal velocity along the interface, {{which has already been imposed weakly during integration by parts.}} Here $\gamma_{pen}$ is a penalty coefficient and $d_h$ is the characteristic mesh size. This penalty term was proposed and implemented in \cite{MassPen} in order to improve the enforcement of mass conservation along a fluid-poroelastic interface when using the primal mixed formulation.

\section{Splitting Scheme} \label{Scheme}

\subsection{Time Discretization}
We propose a noniterative, loosely-coupled scheme for solving the multi-layered fluid-poroelastic structure interaction problem described in \cref{Model} composed of two splitting strategies. The first decouples the multilayered poroelastic structure problem (equations \eqref{eq:thick}-\eqref{eq:thin}) from the fluid problem (equation \eqref{eq:fluid}) using a Backward Euler time-splitting method (see Method 3: BEsplit2 in \cite{StokesDarcyMethods}). We then further separate the poroelastic structure problem into its hydrodynamic components (equations \eqref{eq:thick2}-\eqref{eq:thick3} and \eqref{eq:thin2}-\eqref{eq:thin3}) and elastodynamic components (equations \eqref{eq:thick1} and \eqref{eq:thin1}) using the fixed-strain Biot splitting method \cite{BiotSplit,FixedStrain}. Introducing the backward difference operator $d_tf^{n+1} = \frac{f^{n+1}-f^n}{\Delta t}$ and setting the discrete structure velocities $\boldsymbol\xi^{n+1} = d_t\boldsymbol\eta^{n+1}$ and $v^{n+1} = d_t w^{n+1}$, the splitting scheme can be formulated as a sequence of three steps solved successively at each time step. Details of the scheme are shown in  \cref{table:scheme}.

\SetTblrInner{rowsep=0pt}
\begin{table}[h]
\centering
   \begin{tblr}{colspec = {| l |l |} , cell{1,4,7}{1} = {c=2}{c}, row{3,6,9} = {c, rowsep=8pt}}
\hline
\textbf{Step 1:} The pore pressure equations in $\Omega_b\cup\Omega_p$ & \\
\hline
 \scriptsize PDE & \scriptsize INTERFACE/BOUNDARY DATA \\
\hline
 \scriptsize $\substack{\displaystyle c_0d_tp^{n+1} + \alpha \nabla\cdot \boldsymbol\xi^{n} + \nabla\cdot\boldsymbol u_b^{n+1} = G_b^{n+1} \\ \\ \\
 \displaystyle \boldsymbol\kappa^{-1}\boldsymbol u_b^{n+1} = -\nabla p^{n+1}
\\ \\ \\ \displaystyle c_0^pd_tq^{n+1} - \alpha^pz\Delta_{\boldsymbol\tau^p} \tilde v^{n} + \nabla_{\boldsymbol n^p}u_p^{n+1} = G_p^{n+1}
\\ \\ \\ \displaystyle (\kappa^p)^{-1}u_p^{n+1} = -\nabla_{\boldsymbol n^p}q^{n+1}}$ 
&\scriptsize $\substack{  \displaystyle \left.q^{n+1}\right|_{\Gamma_{+}}=\left.p^{n+1}\right|_{\Gamma}, \left.u_p^{n+1}\right|_{\Gamma_{+}}=\left.\boldsymbol{u}_b^{n+1}\right|_{\Gamma} \cdot \boldsymbol{n}^{m}
\\ \\ \\ \displaystyle\left.u_p^{n+1}\right|_{\Gamma_{-}}=\left.\boldsymbol{u}^n\right|_{\Gamma} \cdot \boldsymbol{n}^{m}-v^n }$  \\
\hline
 \textbf{Step 2:} The structure displacement equations in $\Omega_b$ and $\Gamma$ & \\
\hline
 \scriptsize PDE & \scriptsize INTERFACE/BOUNDARY DATA \\
\hline 
\scriptsize $\displaystyle \rho_bd_t\boldsymbol\xi^{n+1} - \nabla\cdot\boldsymbol\sigma_b(\boldsymbol\eta^{n+1},p^{n+1}) + \gamma\boldsymbol\eta^{n+1} = \boldsymbol F_b^{n+1}$
&\scriptsize $\substack{\displaystyle 
v^{n+1} = \left.\boldsymbol\xi^{n+1}\right|_\Gamma\cdot \boldsymbol n^p, \quad w^{n+1} = \left.\boldsymbol\eta^{n+1}\right|_\Gamma\cdot \boldsymbol n^p \\ \\ \\
\displaystyle {\color{black}H\rho_pd_tv^{n+1}}   -{\color{black}H^3D \Delta_{\boldsymbol \tau^p} \Lambda^{n+1}} + H\gamma^p w^{n+1}
\\ \displaystyle + {\color{black}H\alpha^p \Delta_{\boldsymbol \tau^p}\overline{zq^{n+1}} } = F_p^{n+1}  %\qquad \qquad \text{ on }\Gamma \quad  
\\ \\ \\\displaystyle \qquad \qquad \Lambda^{n+1} = - \Delta_{\boldsymbol \tau^p}  w^{n+1} 
%\quad \text{ on }\Gamma
\\ \\ \\\displaystyle F_p^{n+1} = \left.\boldsymbol n^p\cdot\boldsymbol\sigma_b(\boldsymbol\eta^{n+1},p^{n+1})\boldsymbol n^p\right|_\Gamma +  \left.q^{n+1}\right|_{\Gamma_-} \\ \\ \\ 
\displaystyle 
\left.\boldsymbol\xi^{n+1}\right|_\Gamma\cdot \boldsymbol \tau^p = 0, \quad \left.\boldsymbol\eta^{n+1}\right|_\Gamma\cdot \boldsymbol \tau^p =0
}$ \\
\hline
 \textbf{Step 3:} The Stokes equations in $\Omega_f$ with stress data on $\Gamma$ & \\
\hline 
 \scriptsize PDE & \scriptsize INTERFACE/BOUNDARY DATA \\
\hline
\scriptsize $\substack{\displaystyle \rho_fd_t\boldsymbol u^{n+1}  - \nabla\cdot\boldsymbol\sigma_f(\boldsymbol u^{n+1},\pi^{n+1}) = \boldsymbol F_f^{n+1}
 \\ \\ \\\displaystyle \nabla \cdot\boldsymbol u^{n+1} = 0}$ 
& \scriptsize $\substack{\displaystyle {\color{black}\left.\boldsymbol\tau^p\cdot\boldsymbol\sigma_f(\boldsymbol u^{n+1},\pi^{n+1})\boldsymbol n^p\right|_\Gamma} = {\color{black}-\left.\beta\boldsymbol\tau^p\cdot\boldsymbol u^{n+1}\right|_\Gamma}
 \\ \\ \\ \displaystyle {\color{black}\left.\boldsymbol n^p\cdot\boldsymbol\sigma_f(\boldsymbol u^{n+1},\pi^{n+1})\boldsymbol n^p\right|_\Gamma} ={\color{black} -\left. q^{n+1}\right|_{\Gamma_-}}}$  \\  
\hline
\end{tblr}
\caption{Semidiscrete splitting scheme}
\label{table:scheme}
\end{table}

\subsection{Computational Domain}

Since the physical domains $\Omega_b$, $\Omega_p$, and $\Omega_f$ overlap, we introduce nonoverlapping computational domains  $\hat\Omega_b$, $\hat\Omega_p$, and $\hat\Omega_f$ obtained by gluing the upper boundary $\Gamma_+$ of $\Omega_p$ to the lower boundary $\Gamma$ of $\Omega_b$, and similarly gluing the lower boundary $\Gamma_-$ of $\Omega_p$ to the upper boundary $\Gamma$ of $\Omega_f$ (see \cref{fig:CompDomain}). This has the effect of sliding $\Omega_p$ in the negative $z$ direction by $\frac{H}{2}$ and $\Omega_f$ by $H$. All variables in the computational domain defined on a codimension 1 interior subdomain are therefore defined either on $\hat \Gamma_+$ or $\hat \Gamma_-$; in particular, traces onto $\Gamma$ of functions on $\Omega_b$ are written as traces onto $\hat\Gamma_+$ of functions on $\hat\Omega_b$ and similarly for functions on $\Omega_f$. This results in $\hat w, \hat v$, and $\hat \Lambda$ being defined on $\hat\Gamma_+$.

With this in mind, recalling the definition of the extension operator \eqref{extension}, the coupling conditions \eqref{CC} restated in the computational domain read as follows
\begin{align*}
\left\{\begin{array}{lr}
 q =  p & \text{ in } \hat\Gamma_+\times(0,T), \\
\boldsymbol u_b \cdot\boldsymbol n^p=u_p & \text{ in } \hat\Gamma_+\times(0,T), \\
\left.\boldsymbol\eta\right|_{\hat\Gamma_+} =  w \boldsymbol n^p & \text{ in } \hat\Gamma_+\times(0,T), \\
F^p = \boldsymbol\sigma_b(\boldsymbol \eta,p)\boldsymbol n^p \cdot \boldsymbol n^p - \left.{\boldsymbol\sigma}_f(\widetilde{\boldsymbol u},\widetilde{\pi})\boldsymbol n^p\right|_{\hat\Gamma_-} \cdot \boldsymbol n^p & \text{ in } \hat\Gamma_+\times(0,T), \\
\boldsymbol\sigma_f(\boldsymbol u,\pi)\boldsymbol n^p \cdot \boldsymbol\tau^p_i = -\beta\boldsymbol u\cdot \boldsymbol\tau^p_i & \text{ in } \hat\Gamma_-\times(0,T), \\
q = -\boldsymbol\sigma_f(\boldsymbol u,\pi)\boldsymbol n^p \cdot \boldsymbol n^p & \text{ in } \hat\Gamma_-\times(0,T), \\
u_p = \boldsymbol u \cdot\boldsymbol n^p - \tilde v & \text{ in } \hat\Gamma_-\times(0,T). \\
\end{array}\right.  
\end{align*}
The advantage of our splitting scheme is that now the extension and averaging operations can be calculated in between each step as opposed to occurring implicitly in the problem. In addition, we can see that decoupling the pressure and displacement variables results in two simpler subproblems. The pressures are continuous and have continuous normal derivatives across $\Gamma_+$ and so can be thought of as a single variable $P \in H^1(\hat\Omega_b\cup\hat\Omega_p)$ satisfying an elliptic equation of the form $Ad_tP = \nabla\cdot(B\nabla P) + C$ with discontinous coefficients $A, B, C$.

\tikzmath{\RR = 2; \LL =2; \HH = .3;\Sep=2;} 

\begin{figure}[h]
\centering
\scalebox{.9}{\begin{tikzpicture}[fill opacity=0.3,text opacity=1]
\fill[line width=1.8pt,blue] (0,-\RR) rectangle (\LL,0) node[midway] {\Large\color{black}$\Omega_f$};
\draw[line width=1.8pt,green!70!blue!] (0,-\RR)  node[shift={(-0.5,\RR/2)}] {\Large $\Gamma_f$} rectangle (\LL,0);
\fill[line width=1.8pt, red] (0,0)  rectangle (\LL,\RR) node[midway] {\Large\color{black}$\Omega_b$};
\draw[line width=1.8pt, orange] (0,0) node[shift={(-0.5,\RR/2)}] {\Large $\Gamma_b$} rectangle (\LL,\RR);
\draw[line width=2.5pt,red!50!blue!] (0,0) node[shift={(-0.5,0)}] {\Large $\Gamma$} -- (\LL,0);
\fill[line width=1.8pt,red!50!blue!] (0,-.2) rectangle (\LL, .2);

\draw[line width=3pt,opacity = .5,double,-Stealth] (\LL+.3,0)  -- (\LL+\Sep-.6,0) ;

\fill[line width=1.8pt,blue] (\LL+\Sep,-2*\HH-\RR) rectangle (2*\LL+\Sep,-2*\HH) node[midway] {\Large\color{black}$\hat\Omega_f$};
\draw[line width=1.8pt,green!70!blue!] (\LL+\Sep,-2*\HH-\RR)  node[shift={(-0.5,\RR/2)}] {\Large $\hat \Gamma_f$} rectangle (2*\LL+\Sep,-2*\HH);
\fill[line width=1.8pt, red] (\LL+\Sep,0)  rectangle (2*\LL+\Sep,\RR) node[midway] {\Large\color{black}$\hat\Omega_b$};
\draw[line width=1.8pt, orange] (\LL+\Sep,0) node[shift={(-0.5,\RR/2)}] {\Large $\hat \Gamma_b$} rectangle (2*\LL+\Sep,\RR);
\draw[line width=1pt] (2*\LL+\Sep,-2*\HH)  -- (2*\LL+\Sep,0);
\draw[line width=1pt] (\LL+\Sep,-2*\HH)  -- (\LL+\Sep,0);
\draw[line width=1.8pt,red!80!blue!] (\LL+\Sep,0)  node[shift={(-0.25,0)}] {\color{black} $0$} -- (2*\LL+\Sep,0) node[right] { $\hat \Gamma_+$};
\draw[line width=1.8pt,red!30!blue!] (\LL+\Sep,-2*\HH)  node[shift={(-0.4,0)}] {\color{black} $-H$} -- (2*\LL+\Sep,-2*\HH) node[right] { $\hat \Gamma_-$};
\fill[line width=1.8pt,red!50!blue!] (\LL+\Sep,-2*\HH) rectangle (2*\LL+\Sep,0) node[midway] {\color{black} $\hat \Omega_p$};

\end{tikzpicture}}
\qquad
\scalebox{.9}{\begin{tikzpicture}[line width= 1.4pt]

    % Grid
    %\draw[gray!30, thin] (-2,-2) grid (2,2);

    % Axes
    \draw[-Stealth] (0,0)--(1.2*\RR,0) node[right] {$\hat z$};
    \draw[-Stealth] (0,0)--(0,1.2*\RR) node[left] {$z$};

    % Domains
 	\draw[blue] (-2*\HH-\RR,0)--(-2*\HH,0) node[shift={(-3*\RR/4,-.4)}] { $\hat\Omega_f$};
	\draw[red!50!blue!] (-2*\HH,0)--(0,0) node[shift={(-\HH,-.4)}] { $\hat\Omega_p$};
	\draw[red] (0,0)--(\RR,0) node[shift={(-\RR/4,-.4)}] { $\hat\Omega_b$};
	\draw[blue] (0,-\RR) node[shift={(-0.3,\RR/4)}] { $\Omega_f$} --(0,0) ;
	\draw[red] (0,0) --(0,\RR) node[shift={(-0.3,-\RR/4)}] { $\Omega_b$};
	%\draw[red!35!blue!] (0,-.3)--(0,0);
	%\draw[red!65!blue!] (0,0)--(0,.3);
	\draw[red!50!blue!,opacity=.5,line width=2.7] (0,-.3)--(0,.3);
    % Dashed extensions
    \draw[dashed, blue] (-2*\HH-\RR,-\RR) -- (-2*\HH,0);
    \draw[dashed, red] (0,0) -- (\RR,\RR);
    \draw[dashed, red!50!blue!] (-2*\HH,-.3) -- (0,.3);

\end{tikzpicture}}

\caption{Discontinuous translation of the physical domain to the computational domain. {{The vertical axis shows the original, overlapping domains, while the horizontal axis shows the translated computational domains. }}}
\label{fig:CompDomain}
\end{figure}
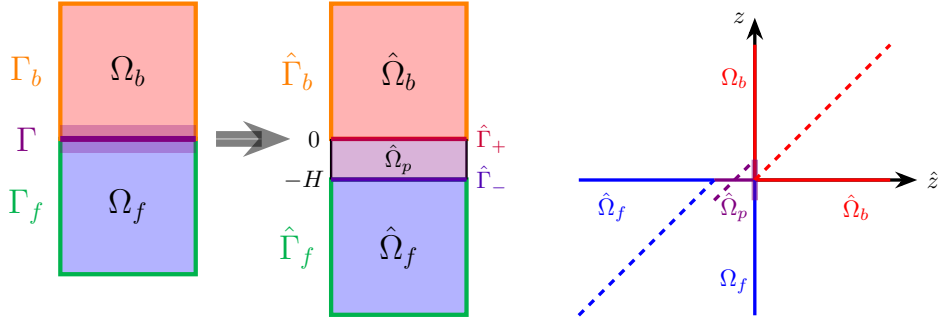

\subsection{Weak Formulation}

We specify the weak form on the appropriate domain but suppress the hat notation going forward. Note that all integrals involving $\Gamma$ are now evaluated over either $\Gamma_+$ or $\Gamma_-$. As in the continuous problem, penalty terms are added to improve mass conservation at the interface, which is known to be insufficiently accurate when primal or mixed primal formulations of the Biot problem are used (see \cite{MassPen}). {{The scheme in weak form reads:}}

Step 1: Find $(p^{n+1}, \boldsymbol u_b^{n+1}, q^{n+1}, u_p^{n+1}) \in Q_s$ so that for every $(r, \boldsymbol U_b, s, U_p) \in Q_s$:
\begin{equation*} \begin{split} 
  &  c_0 (d_t p^{n+1},r)_{\Omega_b} - ( \boldsymbol u_b^{n+1},\nabla r)_{\Omega_b} +  ( \nabla p^{n+1},\boldsymbol U_b)_{\Omega_b}  +  ( \kappa^{-1}\boldsymbol u_b^{n+1},\boldsymbol U_b)_{\Omega_b}   \\
&  + c_0^p (d_t q^{n+1},s)_{\Omega_p} - (  u_p^{n+1},\nabla_{\boldsymbol n^p} s)_{\Omega_p} +  ( \nabla_{\boldsymbol n^p} q^{n+1}, U_p)_{\Omega_p}  +  ( (\kappa^p)^{-1} u_p^{n+1}, U_p)_{\Omega_p}  \\
& =  (G_b^{n+1},r)_{\Omega_b}  + (G_p^{n+1},s)_{\Omega_p}  -  \alpha(\nabla\cdot\boldsymbol\xi^{n},r)_{\Omega_b} -  \alpha^p(\nabla_{\boldsymbol\tau^p}\tilde v^{n},\left(z + H/2\right)\nabla_{\boldsymbol\tau^p}s)_{\Omega_p}   \\
&+ \left(\boldsymbol{u}^n \cdot \boldsymbol{n}^{p}-\tilde v^n,s  \right)_{\Gamma_-} +   \frac{\gamma_{pen}\mu_f}{d_h}\left(\boldsymbol u^n\cdot\boldsymbol n^p - \tilde v^n - u_p^{n+1}, U_p\right)_{\Gamma_-}
\end{split}  \end{equation*}

Step 2: Find $(\boldsymbol\xi^{n+1}, v^{n+1},\Lambda^{n+1}) \in V_s$ with $d_t\boldsymbol\eta = \boldsymbol \xi, d_t w = v$ so that for every $(\boldsymbol \phi, \varphi, \psi) \in V_s$:
\begin{equation*} \begin{split} 
&  \rho_b(d_t \boldsymbol\xi^{n+1},\boldsymbol\phi)_{\Omega_b}  + \lambda_b(\nabla\cdot\boldsymbol\eta^{n+1},\nabla\cdot\boldsymbol\phi)_{\Omega_b} + 2\mu_b(\boldsymbol D(\boldsymbol\eta^{n+1}),\boldsymbol D(\boldsymbol\phi))_{\Omega_b} + \gamma (\boldsymbol\eta^{n+1},\boldsymbol\phi)_{\Omega_b} \\
& + H\rho_p(d_t v^{n+1},\varphi)_{\Gamma_+} + H^3D(\nabla_{\boldsymbol\tau^p}\Lambda^{n+1},\nabla_{\boldsymbol\tau^p}\varphi )_{\Gamma_+} + H^3D(d_t\Lambda^{n+1},\psi)_{\Gamma_+}  \\
&- H^3D(\nabla_{\boldsymbol\tau^p}v^{n+1},\nabla_{\boldsymbol\tau^p}\psi )_{\Gamma_+}   + H\gamma^p(w^{n+1},\varphi)_{\Gamma_+} {\color{black} = }  (\boldsymbol F_b^{n+1},\boldsymbol\phi)_{\Omega_b} +  \alpha (p^{n+1},\nabla\cdot\boldsymbol\phi)_{\Omega_b}   \\
& + H\alpha^p\left( \nabla_{\boldsymbol\tau^p}\overline{\left(z + H/2\right)q^{n+1}},\nabla_{\boldsymbol\tau^p}\varphi \right)_{\Gamma_+}  +   \left(\widetilde{\left.q^{n+1}\right|_{\Gamma_-}} ,\varphi  \right)_{\Gamma_+} \\
 &  +\frac{\gamma_{pen}\mu_f}{d_h}\left(\widetilde{\left.\boldsymbol u^n\right|_{\Gamma_-}}\cdot\boldsymbol n^p - v^{n+1} - \widetilde{\left.u_p^{n+1}\right|_{\Gamma_-}}, \varphi \right)_{\Gamma_+}
\end{split}  \end{equation*}

Step 3: Find $(\boldsymbol u^{n+1}, \pi^{n+1}) \in V_f$ so that for every $(\boldsymbol U, \Pi) \in V_f$:
\begin{equation*} \begin{split} 
& \rho_f(d_t\boldsymbol u^{n+1},\boldsymbol U)_{\Omega_f} + 2\mu_f(\boldsymbol D(\boldsymbol u^{n+1}),\boldsymbol D(\boldsymbol U))_{\Omega_f} - (\pi^{n+1},\nabla\cdot\boldsymbol U)_{\Omega_f} +  (\nabla\cdot\boldsymbol u^{n+1},\Pi)_{\Omega_f}   \\
&  + \beta\left(\boldsymbol P_{\boldsymbol{\tau}^p} \boldsymbol{u},\boldsymbol P_{\boldsymbol{\tau}^p} \boldsymbol{U}\right)_{\Gamma_-} = (\boldsymbol F_f^{n+1},\boldsymbol U)_{\Omega_f} - \left(q^{n+1},\boldsymbol{U} \cdot \boldsymbol{n}^{p} \right)_{\Gamma_-} \\
&   - \frac{\gamma_{pen}\mu_f}{d_h}\left(\boldsymbol u^{n+1}\cdot\boldsymbol n^p - \tilde v^{n+1} - u_p^{n+1},\boldsymbol U\cdot\boldsymbol n^p \right)_{\Gamma_-}
\end{split}  \end{equation*}

\section{Stability} \label{Stability}

In this section, we analyze the stability of the partitioned scheme in the absence of the penalty term. We start by discretizing in space using the finite element method with a conforming finite element triangulation $\mathcal T_h$ of $\Omega$, where $h$ is the maximum diameter of an element in $\mathcal T_h$. For a simplex $K \in \mathcal T_h$, let $h_K$ be the diameter of $K$ and $\rho_K$ the diameter of the inscribed sphere. We assume that the family of triangulations $\mathcal T_h$ satisfies the following usual inverse assumption in $\Omega_b$ 
and has the following shape regularity property in $\Omega_b\cup\Omega_f$  \cite{ciarlet_finite_2002,ern_finite_2021}:
\begin{equation}\label{InverseAndShape}
\sup_K\frac{h}{h_K} < M_{INV} \ \text{in}\  \Omega_b,\  \text{and}\ 
\sup_K\frac{h_K}{\rho_K} < M_{SR}\  \text{in}\  \Omega_b\cup\Omega_f
\end{equation}
for some constants $M_{INV}$ and $M_{SR}$.
We specifically do not assume shape regularity in the thin domain $\Omega_p$ since in practice this is a stricter condition on the mesh due to the anisotropy. We define the finite elements spaces $\mathcal Q^b_h \subset H_0^1(\Omega_b), \mathcal W^b_h \subset \boldsymbol L^2(\Omega_b), \mathcal Q^p_h \subset H^1(\Omega_p), \mathcal W^p_h \subset  L^2(\Omega_p), \mathcal V^{b}_h \subset \boldsymbol H^1_0(\Omega_b), \mathcal W^s_h \subset H^1(\Gamma), \mathcal W^f_h \subset \boldsymbol H_0^1(\Omega_f), \mathcal Q^f_h \subset  L^2(\Omega_f)$ and $\mathcal V^{p}_h = \left.\mathcal V^{b}_h\right|_{\Gamma^+}$ and assume that the pairs $\mathcal V^{p}_h, \mathcal W^{s}_h$ and $\mathcal W^{f}_h, Q^f_h$ are inf-sup stable. We thus have the corresponding mixed finite element spaces:
\begin{equation*} \begin{split} 
Q_h^s & = \left\{(p_h, \boldsymbol u_{b,h}, q_h, u_{p,h}) \in \mathcal Q^b_h\times \mathcal W_h^b\times\mathcal  Q^p_h\times \mathcal W_h^p : \left.q_h\right|_{\Gamma_{+}}=\left.p_h\right|_{\Gamma_+} \right\} \subset \mathcal Q^s, \\
\mathcal V_h^s & = \mathcal V_h^b\times \mathcal W_h^s\subset \mathcal V^s, \\
 \mathcal V_h^f &= \mathcal W_h^f\times \mathcal Q_h^f\subset \mathcal V^f.
\end{split} \end{equation*}

\subsection{Preliminaries}
In the thick structure, will make use of the following standard inequalities for $p \in H^1_0(\Omega_b)$ (here all the norms refer to the $L^2$ norms):
\begin{equation*}  \text{Poincar\'{e} inequality:} \quad ||p||_{\Omega_b} \leq C_{P}||\nabla p||_{\Omega_b} \leq \frac{C_P}{\sqrt{k_{min}}}||\boldsymbol\kappa^{-1/2}\boldsymbol u_b||_{\Omega_b}. \end{equation*} 
\begin{equation*}  \text{Trace inequality:} \quad ||p||_{\Gamma_+} \leq  C_{TR}||\nabla p||_{\Omega_b} \leq  \frac{C_{TR}}{\sqrt{k_{min}}}||\boldsymbol\kappa^{-1/2}\boldsymbol u_b||_{\Omega_b}. \end{equation*} 
Due to the inverse assumption in $\Omega_b$ specified in \eqref{InverseAndShape} we have the inverse inequality (see Theorem 3.2.6 in \cite{ciarlet_finite_2002})
\begin{equation*}  ||\nabla\cdot \boldsymbol\xi_h||_{\Omega_b} \leq\sqrt{3} ||\nabla \boldsymbol\xi_h||_{\Omega_b} \leq \frac{C_{INV}\sqrt{3}}{h}||\boldsymbol\xi_h||_{\Omega_b} \end{equation*} 
for every $\xi_h \in \mathcal V_h^b$. Furthermore, since $\mathcal V^{p}_h = \left.\mathcal V^{b}_h\right|_{\Gamma^+}$, 
we also have
\begin{equation*}  ||\nabla_{\pmb\tau^p} v_h||_{\Gamma_+} \leq \frac{B_{INV}}{h}||v_h||_{\Gamma_+} \end{equation*} 
for every $v_h \in \mathcal V^{p}_h$. In the thick Biot and fluid domains, shape regularity gives for $\xi_h \in \mathcal V_h^b$ and $\boldsymbol u_h \in \mathcal W_h^f$ the discrete trace inverse inequalities (see Lemma 12.8 in \cite{ern_finite_2021})
\begin{equation*}   ||v_h||_{\Gamma_+} \leq \frac{C_{TI}}{\sqrt{h}}||\boldsymbol \xi_h||_{\Omega_b}, \qquad\qquad ||\boldsymbol u_h||_{\Gamma_-} \leq \frac{A_{TI}}{\sqrt{h}}||\boldsymbol u_h||_{\Omega_f}. \end{equation*} 
Since we do not assume shape regularity in the thin domain, the following theorem is used as a suitable alternative to the discrete trace inverse inequality for our purposes.
\begin{theorem} \label{thm:Poinc}
The pressure $q\in H^1(\Omega_p)$ satisfies the trace inequality
\begin{equation*}  ||q||_{\Gamma_-} \leq  \frac{C_{TR}\sqrt{2}}{\sqrt{k_{min}}}||\boldsymbol\kappa^{-1/2}\boldsymbol u_b||_{\Omega_b} +  \sqrt{\frac{2H}{\kappa^p}}||(\kappa^p)^{-1/2}u_p||_{\Omega_p} \end{equation*} 
and the average inequality
\begin{equation*}  \left|\left|\overline{\left(z+H/2\right)q}\right|\right|_{\Gamma_+} \leq \frac{H^{3/2}}{8\sqrt{\kappa^p}}||(\kappa^p)^{-1/2}u_p||_{\Omega_p}. \end{equation*} 
\end{theorem}
\begin{proof}
We have
\begin{equation*}  q\left(x,y,-H\right) = p\left(x,y,0\right) - \int_{-H}^{0}\nabla_{\boldsymbol n^p} q(x,y,s) \ ds \end{equation*} 
and so by Cauchy-Schwartz
\begin{equation*}  q\left(x,y,-H\right)^2 \leq 2p\left(x,y,0\right)^2 + 2H\int_{-H}^{0}\left|\nabla_{\boldsymbol n^p} q(x,y,s) \right|^2 \ ds. \end{equation*} 
Integrating over $x$ and $y$ then yields the trace inequality
\begin{equation*}  ||q||_{\Gamma_-}^2 \leq 2 ||p||_{\Gamma_+}^2 + 2H||\nabla_{\boldsymbol\tau^p} q||_{\Omega_p}^2 \leq \frac{2C_{TR}^2}{k_{min}}||\boldsymbol\kappa^{-1/2}\boldsymbol u_b||_{\Omega_b}^2 +  \frac{2H}{\kappa^p}||(\kappa^p)^{-1/2}u_p||_{\Omega_p}^2. \end{equation*} 
Similarly, for any $x, y$, we use integration by parts to see that $\overline{\left(z+H/2\right)q}(x,y,0)$ is given by
\begin{equation*} \frac{1}{H}\int_{-H}^{0}\left(z+\frac{H}{2}\right)q(x,y,z) dz = -\frac{1}{2H}\int_{-H}^{0}\left(z^2+Hz\right)\nabla_{\boldsymbol n^p}q(x,y,z) dz \end{equation*} 
and so
\begin{equation*}  \left|\overline{\left(z+H/2\right)q}(x,y,0)\right| \leq  \int_{-H}^{0}\left|\frac{z^2+Hz}{-2H}\right| \left|\nabla_{\boldsymbol n^p}q(x,y,z)\right|dz \leq \frac{H}{8}\int_{-H}^{0}\left|\nabla_{\boldsymbol n^p}q(x,y,z)\right|dz.   \end{equation*} 
Squaring both sides, applying Cauchy-Schwarz, and then integrating over $x, y$ yields 
\begin{equation*} \left|\left|\overline{\left(z+H/2\right)q}\right|\right|_{\Gamma_+}^2 \leq \frac{H^3}{64}\left|\left|\nabla_{\boldsymbol n^p}q\right|\right|_{\Omega_p}^2 \leq \frac{H^3}{64\kappa^p}\left|\left|(\kappa^p)^{-1/2}u_p\right|\right|_{\Omega_p}^2. \end{equation*} 
\end{proof}

We will also repeatedly make use of the following lemmas.
\begin{lemma} \label{lem:1} Suppose $C^2 < AB$ for $A, B > 0$. Then there exist positive constants $\hat A, \hat B > 0$ such that $\hat Ax_1^2 + \hat Bx_2^2 \leq Ax_1^2 + Bx_2^2 + 2Cx_1x_2$ for all $x_1,x_2$.
\end{lemma}

\begin{proof} This follows from the fact that $Ax_1^2 + Bx_2^2 + 2Cx_1x_2$ is the quadratic form represented by the positive definite matrix $\left(\begin{matrix} A & C \\ C & B \end{matrix}\right)$ and so there exists some diagonal positive definite quadratic form dominated by it. More constructively, we use the fact that there exists some $\epsilon > 0$ such that $C^2 < \epsilon < AB$ and let $\hat A = A - \frac{\epsilon}{B}$ and $\hat B := B - \frac{BC^2}{\epsilon}$. Then $\hat A, \hat B > 0$ and by Young's inequality we have
\begin{equation*}   2|Cx_1x_2| = 2\left(|x_1|\sqrt{\frac{\epsilon}{B}}\right)\left(|Cx_2|\sqrt{\frac{B}{\epsilon}}\right) \leq \frac{\epsilon}{B}x_1^2 + \frac{BC^2}{\epsilon}x_2^2, \end{equation*} 
and thus $Ax_1^2 + Bx_2^2 + 2Cx_1x_2 \geq \left(A - \frac{\epsilon}{B}\right)x_1^2 + \left(B - \frac{BC^2}{\epsilon}\right)x_2^2 = \hat Ax_1^2 + \hat Bx_2^2$.
\end{proof}

\begin{lemma} \label{lem:2} Let $A = \left\{A_{ij}\right\}$ be an $n\times n$ symmetric matrix with positive diagonal entries $A_{11}, A_{22}, \ldots, A_{nn}$. If $A_{ij}^2 < \frac{1}{(n-1)^2}A_{ii}A_{jj}$ for all $i \neq j$, then $A$ is positive definite and there exists a diagonal matrix $D = \operatorname{diag}(A_1,A_2,...,A_n)$, $0 < A_i < A_{ii}$ for $i = 1,\ldots, n$, such that $A \geq D$.
\end{lemma}

\begin{proof}
This follows from Sylvester's law of inertia after applying the Gershgorin circle theorem to the matrix $SAS^T$ for $S = \operatorname{diag}(A_{11},A_{22},...,A_{nn})^{-1/2}$.
\end{proof}

\begin{remark}
For any positive definite matrix $A = \left\{A_{ij}\right\}$, one can readily show that $A_{ij}^2 < A_{ii}A_{jj}$ for $i \neq j$ by considering the $2\times 2$ principle minors $A_{ii}A_{jj} - A_{ij}^2 > 0$. This lemma can be seen as a type of converse; if each of the $2\times 2$ principle submatrices satisfy the stronger condition $A_{ij}^2 < \frac{1}{(n-1)^2}A_{ii}A_{jj}$ for $i\neq j$, then $A$ can be shown to be positive definite.
\end{remark}

\begin{remark}
If the matrix $A$ is known to have some zero entries, then it is actually sufficient to show positive definiteness if $A_{ij}^2 < \frac{1}{n_in_j}A_{ii}A_{jj}$ for all $i \neq j$ where $n_i$ is the number of nonzero entries off the diagonal in row $i$. Furthermore, if $A_{ij}$ is a sum $\sum_{k=1}^l A_{ij,k}$, then it is sufficient to show $A_{ij,k}^2 < \frac{1}{l^2n_in_j}A_{ii}A_{jj}$ for all $i \neq j$ and $k=1,\ldots,l$.
\end{remark}

\subsection{Stability Results}

\begin{theorem} \label{thm:stab1}
Assume that the fluid-multilayered poroelastic system is isolated, that is, $\boldsymbol F_b = 0, G_b = 0, G_m = 0$, and $\boldsymbol F_f =  0$. Let $M = \min\left\{\frac{\rho_b}{C_{TI}^2},\frac{\rho_f}{A_{TI}^2}\right\}$. Then under the condition
\begin{equation*}  \Delta t < \min\left\{\frac{\rho_bk_{min}}{24\alpha^2C_{INV}^2C_{PF}^2}h^2,\frac{128\rho_p\kappa^p}{3(\alpha^p)^2B_{INV}^4}\left(\frac{h}{H}\right)^4,\frac{Mk_{min}}{16C_{TR}^2}h,\frac{M\kappa^p}{6}\left(\frac{h}{H}\right)\right\}, \end{equation*} 
the following {{stability}} estimate holds for some constants $\epsilon_{\boldsymbol\xi}, \epsilon_{v}, \epsilon_{\boldsymbol u}, \epsilon_{\boldsymbol u_b}, \epsilon_{u_p} \in (0,1)$:
\begin{align*}\small \begin{split} & \mathcal E(\boldsymbol A_h^{n+1}) + \epsilon_{\boldsymbol\xi}\rho_b||\boldsymbol\xi_h^{n+1} - \boldsymbol\xi_h^{n}||^2_{\Omega_b} + \lambda_b||\nabla\cdot(\boldsymbol\eta_h^{n+1}-\boldsymbol\eta_h^{n})||^2_{\Omega_b} + 2\mu_b||\boldsymbol D(\boldsymbol\eta_h^{n+1}-\boldsymbol\eta_h^{n})||^2_{\Omega_b} \\
& +  \gamma||\boldsymbol\eta_h^{n+1}-\boldsymbol\eta_h^{n}||^2_{\Omega_b}  + \epsilon_{v}\rho_p||v_h^{n+1} - v_h^n||^2_{\Gamma_+} + H^3D||\Lambda_h^{n+1} - \Lambda_h^n||^2_{\Gamma_+} + H\gamma^p||w_h^{n+1}-w_h^n||^2_{\Gamma_+} \\
& +  c_0||p_h^{n+1}-p_h^n||^2_{\Omega_b} + c_0^p||q_h^{n+1}-q_h^n||^2_{\Omega_p} + \epsilon_{\boldsymbol u}\rho_f||\boldsymbol u_h^{n+1}-\boldsymbol u_h^{n}||^2_{\Omega_f} + \epsilon_{\boldsymbol u_b}2\Delta t||\boldsymbol\kappa^{-1/2}\boldsymbol u_{b,h}^{n+1}||^2_{\Omega_b} \\
&+ \epsilon_{u_p}2\Delta t||(\kappa^p)^{-1/2} u_{p,h}^{n+1}||^2_{\Omega_p} + 4\Delta t \mu_f||\boldsymbol D (\boldsymbol u_h^{n+1})||^2_{\Omega_f} + 2\beta\Delta t \left|\left|\boldsymbol P_{\boldsymbol\tau^p}\boldsymbol{u}_h^{n+1}\right|\right|^2_{\Gamma_-} \leq  \mathcal E(\boldsymbol A_h^{n}).
\end{split}  \end{align*}
\end{theorem}

\begin{proof}

In the following, we take all force terms to be zero and repeatedly use the polarization identity $2a(a-b) = a^2 + (a-b)^2 - b^2$. Setting $(r_h, \boldsymbol U_{b,h}, s_h, U_{p,h}) = 2\Delta t(p_h^{n+1}, \boldsymbol u_{b,h}^{n+1}, q_h^{n+1}, u_{p,h}^{n+1})$ in the weak form of Step 1 we get
\begin{equation*}
\small \begin{split} 
&  c_0\left(||p_h^{n+1}||_{\Omega_b}^2 + ||p_h^{n+1} - p_h^{n}||_{\Omega_b}^2\right)+ c_0^p\left(||q_h^{n+1}||_{\Omega_p}^2 + ||q_h^{n+1} - q_h^{n}||_{\Omega_p}^2\right) + 2\Delta t ||\boldsymbol\kappa^{-1/2}\boldsymbol u_{b,h}^{n+1}||^2_{\Omega_b}\\
&  + 2\Delta t ||(\kappa^p)^{-1/2} u_{p,h}^{n+1}||^2_{\Omega_p} =  c_0||p_h^{n}||_{\Omega_b}^2 + c_0^p||q_h^{n}||_{\Omega_p}^2 -  2\Delta t\alpha(\nabla\cdot\boldsymbol\xi_h^{n},p_h^{n+1})_{\Omega_b} \\  
&-  2\Delta t \left[\alpha^p(\nabla_{\boldsymbol\tau^p}\tilde v_h^{n},\left(z + H/2\right)\nabla_{\boldsymbol\tau^p}q_h^{n+1})_{\Omega_p} - \left(\left.\boldsymbol{u}_h^n\right|_{\Gamma} \cdot \boldsymbol{n}^{p}-\tilde v_h^n,q_h^{n+1}\right)_{\Gamma_-}\right].
\end{split} 
\end{equation*}

In the weak form of Step 2, setting $(\boldsymbol \phi_h, \varphi_h, \psi_h) = 2\Delta t(\boldsymbol\xi_h^{n+1}, v_h^{n+1},\Lambda_h^{n+1})$ we get

\begin{equation*}
\small\begin{split} 
& \rho_b \left(||\boldsymbol\xi_h^{n+1}||_{\Omega_b}^2 +||\boldsymbol\xi_h^{n+1} - \boldsymbol\xi_h^{n}||_{\Omega_b}^2\right) +  \lambda_b\left(||\nabla\cdot\boldsymbol\eta_h^{n+1}||_{\Omega_b}^2 + ||\nabla\cdot(\boldsymbol\eta_h^{n+1} - \boldsymbol\eta_h^{n})||_{\Omega_b}^2\right)   \\  
&+  2\mu_b\left(||\boldsymbol D(\boldsymbol\eta_h^{n+1})||_{\Omega_b}^2 + ||\boldsymbol D(\boldsymbol\eta_h^{n+1} - \boldsymbol\eta_h^{n})||_{\Omega_b}^2\right)  + \gamma\left(|| \boldsymbol\eta_h^{n+1}||_{\Omega_b}^2  + ||\boldsymbol\eta_h^{n+1} - \boldsymbol\eta_h^{n}||_{\Omega_b}^2\right)  \\  
&  + H\rho_p\left(|| v_h^{n+1}||_{\Gamma_+}^2  + ||v_h^{n+1} - v_h^{n}||_{\Gamma_+}^2\right) + H^3D\left(||\Lambda_h^{n+1}||_{\Gamma_+}^2 + ||\Lambda_h^{n+1} - \Lambda_h^n||_{\Gamma_+}^2\right) \\  
& +  H\gamma^p\left(|| w_h^{n+1}||_{\Gamma_+}^2  + ||w_h^{n+1} - w_h^{n}||_{\Gamma_+}^2\right)   = \rho_b||\boldsymbol\xi_h^{n}||_{\Omega_b}^2 + \lambda_b||\nabla\cdot\boldsymbol\eta_h^{n}||_{\Omega_b}^2+ 2\mu_b||\boldsymbol D(\boldsymbol\eta_h^{n})||_{\Omega_b}^2 \\
&   + \gamma||\boldsymbol\eta_h^{n}||_{\Omega_b}^2 + H\rho_p||v_h^n||_{\Gamma_+}^2 + H^3D||\Lambda_h^{n}||_{\Gamma_+}^2 + H\gamma^p||w_h^n||_{\Gamma_+}^2+ 2\alpha \Delta t (p_h^{n+1},\nabla\cdot\boldsymbol\xi_h^{n+1})_{\Omega_b}   \\  
& + 2 \Delta t \left[ H\alpha^p\left( \nabla_{\boldsymbol\tau^p}\overline{(z + H/2)q_h^{n+1}},\nabla_{\boldsymbol\tau^p}v_h^{n+1} \right)_{\Gamma_+}   +   \left(q_h^{n+1} , v_h^{n+1}  \right)_{\Gamma_+}\right].
\end{split} 
\end{equation*}

In the weak form of Step 3, setting $(\boldsymbol U_h, \Pi_h) = 2\Delta t(\boldsymbol u_h^{n+1}, \pi_h^{n+1})$ we get

\begin{equation*}
\begin{split} 
& \rho_f \left(||\boldsymbol u_h^{n+1}||_{\Omega_f}^2 +||\boldsymbol u_h^{n+1} - \boldsymbol u_h^{n}||_{\Omega_f}^2\right) + 2\Delta t\left[2\mu_f||\boldsymbol D(\boldsymbol u_h^{n+1})||_{\Omega_f}^2 + \beta||\boldsymbol P_{\boldsymbol\tau^p}{\boldsymbol u_h}^{n+1}||_{\Gamma_-}^2   \right] \\  
& = \rho_f||\boldsymbol u_h^{n}||_{\Omega_f}^2  - 2\Delta t \left(q_h^{n+1},\left.\boldsymbol{u}_h^{n+1}\right|_\Gamma \cdot \boldsymbol{n}^{p} \right)_{\Gamma_-}.
\end{split} 
\end{equation*}

Setting $\boldsymbol A_h^n := (p_h^{n}, \boldsymbol u_{b,h}^{n}, q_h^{n}, u_{p,h}^{n},\boldsymbol\xi_h^{n}, v_h^{n},\Lambda_h^{n},\boldsymbol u_h^{n}, \pi_h^{n}) \in \mathcal V_h$ and summing yields
\begin{align*} \begin{split} & \mathcal E(\boldsymbol A_h^{n+1}) + \mathcal E(\boldsymbol A_h^{n+1} - \boldsymbol A_h^n) + 2\Delta t \mathcal D(\boldsymbol A_h^{n+1}) =  \mathcal E(\boldsymbol A_h^{n}) \\
& + 2\Delta t\left[\alpha(p_h^{n+1},\nabla\cdot(\boldsymbol\xi_h^{n+1} - \boldsymbol\xi_h^n))_{\Omega_b}  - \left(q_h^{n+1},\left.\left(\boldsymbol u_h^{n+1} - \boldsymbol u_h^n\right)\right|_\Gamma\cdot\boldsymbol n^p\right)_{\Gamma_-} \right] \\
& + 2\Delta t\left[ H\alpha^p(\nabla_{\boldsymbol\tau^p}\overline{(z + H/2)q_h^{n+1}},\nabla_{\boldsymbol\tau^p}( v_h^{n+1} - v_h^{n}))_{\Gamma_+} + \left(q_h^{n+1}  ,v_h^{n+1} - v_h^n\right)_{\Gamma-} \right].
\end{split}  \end{align*}

We thus have four terms we seek to bound. The first term, $(p_h^{n+1},\nabla\cdot(\boldsymbol\xi_h^{n+1} - \boldsymbol\xi_h^n))_{\Omega_b}$, occurs due to the Biot splitting on the thick domain as described in \cite{BiotSplit} and can be bounded by using the inverse and Poincar\'{e}'s inequalities as follows:
\begin{equation*} \begin{split} \left|\alpha\Delta t(p_h^{n+1},\nabla\cdot(\boldsymbol\xi_h^{n+1} - \boldsymbol\xi_h^n))_{\Omega_b}\right| & \leq \alpha\Delta t|| p_h^{n+1}||_{\Omega_b}||\nabla\cdot(\boldsymbol\xi_h^{n+1} - \boldsymbol\xi_h^{n})||_{\Omega_b} \\ 
&\leq \frac{\alpha C_{PF}C_{INV}\Delta t\sqrt{3}}{h\sqrt{k_{min}}}||\boldsymbol\kappa^{-1/2}\boldsymbol u_{b,h}^{n+1}||_{\Omega_b}||\boldsymbol\xi_h^{n+1} - \boldsymbol\xi_h^{n}||_{\Omega_b} .
\end{split} \end{equation*}

Note that $||\boldsymbol\kappa^{-1/2}\boldsymbol u_{b,h}^{n+1}||^2$ appears in the term $2\Delta t \mathcal D(\boldsymbol A_h^{n+1})$ with a coefficient of $2\Delta t$ and $||\boldsymbol\xi_h^{n+1} - \boldsymbol\xi_h^{n}||_{\Omega_b}^2$ appears in $ \mathcal E(\boldsymbol A_h^{n+1}-\boldsymbol A_h^{n})$ with coefficient $\rho_b$. The term $(\nabla_{\boldsymbol\tau^p}\overline{(z + H/2)q_h^{n+1}},\nabla_{\boldsymbol\tau^p}( v_h^{n+1} - v_h^{n}))_{\Gamma_+}$ arises in a similar manner from the Biot splitting of the poroelastic plate model and is handled analogously. We first have
\begin{equation*} \small \begin{split} \left|(\nabla_{\boldsymbol\tau^p}\overline{(z + H/2)q_h^{n+1}},\nabla_{\boldsymbol\tau^p}( v_h^{n+1} - v_h^{n}))_{\Gamma_+}\right|
&  \leq  \left|\left| \nabla_{\boldsymbol\tau^p}\overline{(z + H/2)q_h^{n+1}}\right|\right|_{\Gamma_+}||\nabla_{\boldsymbol\tau^p}( v_h^{n+1} -  v_h^{n})||_{\Gamma_+}\\
&  \leq \frac{B_{INV}^2}{h^2}\left|\left|  \overline{(z + H/2)q_h^{n+1}}\right|\right|_{\Gamma_+}|| v_h^{n+1} - v_h^{n}||_{\Gamma_+} \end{split} \end{equation*}
from the inverse inequality, which after using the average inequality gives the bound
\begin{equation*} \small \begin{split} \left|H\alpha^p\Delta t(\nabla_{\boldsymbol\tau^p}\overline{(z+H/2)q_h^{n+1}},\nabla_{\boldsymbol\tau^p}( v_h^{n+1} - v_h^{n}))_{\Gamma_+}\right| &
 \\
\leq  \frac{H^{5/2}\alpha^pB_{INV}^2\Delta t}{8h^2\sqrt{\kappa^p}} &||(\kappa^p)^{-1/2}u_{p,h}^{n+1}||_{\Omega_p}|| v_h^{n+1} - v_h^{n}||_{\Gamma_+}.
\end{split} \end{equation*}
We also see that  $||(\kappa^p)^{-1/2}u_{p,h}^{n+1}||_{\Omega_p}^2$ appears in the term $2\Delta t \mathcal D(\boldsymbol A_h^{n+1})$ with a coefficient of $2\Delta t$ and $||v_h^{n+1} - v_h^{n}||_{\Gamma_+}^2$ appears in $ \mathcal E(\boldsymbol A_h^{n+1}-\boldsymbol A_h^{n})$ with coefficient $H\rho_p$.

 The remaining boundary terms occur due to the splitting between the poroelastic domain and the fluid domain as seen in the Stokes-Darcy splitting methods of \cite{StokesDarcyMethods}. The first is bounded by
\begin{equation*} \begin{split} \left|\Delta t \left(q_h^{n+1} ,v_h^{n+1} - v_h^n\right)_{\Gamma_+}\right| 
 \leq & \frac{C_{TI}C_{TR}\Delta t\sqrt{2}}{\sqrt{hk_{min}}}||\boldsymbol\kappa^{-1/2}\boldsymbol u_{b,h}^{n+1}||_{\Omega_b}|| \boldsymbol\xi_h^{n+1} - \boldsymbol\xi_h^{n}||_{\Omega_b} \\
& +   \frac{C_{TI}\sqrt{2H}\Delta t}{\sqrt{h\kappa^p}}||(\kappa^p)^{-1/2}u_{p,h}^{n+1}||_{\Omega_p}|| \boldsymbol\xi_h^{n+1} - \boldsymbol\xi_h^{n}||_{\Omega_b} 
\end{split} \end{equation*}
using the trace inequality for $q_h$ and the discrete trace inverse inequality for $\boldsymbol \xi_h$, and the second by
\begin{equation*} \small\begin{split} \left|\Delta t \left(q_h^{n+1} ,\left.\left(\boldsymbol u_h^{n+1} - \boldsymbol u_h^n\right)\right|_\Gamma\cdot\boldsymbol n^p\right)_{\Gamma_-}\right| 
 \leq &\frac{A_{TI}C_{TR}\Delta t\sqrt{2}}{\sqrt{hk_{min}}}||\boldsymbol\kappa^{-1/2}\boldsymbol u_{b,h}^{n+1}||_{\Omega_b}\left|\left|\boldsymbol u_h^{n+1} - \boldsymbol u_h^{n}\right|\right|_{\Omega_f}  \\
& +  \frac{A_{TI}\sqrt{2H}\Delta t}{\sqrt{h\kappa^p}}||(\kappa^p)^{-1/2}u_{p,h}^{n+1}||_{\Omega_p}\left|\left|\boldsymbol u_h^{n+1} - \boldsymbol u_h^{n}\right|\right|_{\Omega_f}
\end{split}  \end{equation*}
using the trace inequality for $q_h$ and the discrete trace inverse inequality for $\boldsymbol u$. Note that $\left|\left|\boldsymbol u_h^{n+1} - \boldsymbol u_h^{n}\right|\right|_{\Omega_f}^2$ appears in the term $\mathcal E(\boldsymbol A_h^{n+1}-\boldsymbol A_h^{n})$ with coefficient $\rho_f$.

We can now input all of our inequalities in the following matrix representation of our bilinear form on the variable
\begin{equation*} \scriptsize \vec x = \left(\begin{matrix} ||\boldsymbol\kappa^{-1/2}\boldsymbol u_{b,h}^{n+1}||_{\Omega_b} & ||\boldsymbol\xi_h^{n+1} - \boldsymbol\xi_h^{n}||_{\Omega_b}  & ||(\kappa^p)^{-1/2}u_{p,h}||_{\Omega_p} & || v_h^{n+1} - v_h^{n}||_{\Gamma} & \left|\left|\boldsymbol u_h^{n+1} - \boldsymbol u_h^{n}\right|\right|_{\Omega_f} \end{matrix}\right)^T \end{equation*} 
as the symmetric $5\times 5$ matrix
\begin{equation*}\footnotesize \left(\begin{matrix}
2\Delta t & \frac{\alpha C_{PF}C_{INV}\Delta t\sqrt{3}}{h\sqrt{k_{min}}} + \frac{C_{TI}C_{TR}\Delta t\sqrt{2}}{\sqrt{hk_{min}}}  & 0 &  0 & \frac{A_{TI}C_{TR}\Delta t\sqrt{2}}{\sqrt{hk_{min}}} \\
& \rho_b &\frac{C_{TI}\sqrt{2H}\Delta t}{\sqrt{h\kappa^p}}  & 0 & 0 \\
 &  & 2\Delta t &  \frac{H^{5/2}\alpha^pB_{INV}^2\Delta t}{8h^2\sqrt{\kappa^p}} &  \frac{A_{TI}\sqrt{2H}\Delta t}{\sqrt{h\kappa^p}} \\ 
 &  &  & H\rho_p & 0 \\
 &  & & & \rho_f
\end{matrix}\right). \end{equation*} 
Using \cref{lem:2} and bounding each of the terms appearing off the diagonal yields the time-step condition above as a sufficient condition for stability.
\end{proof}

The squared $h$ term in the above stability condition appears as a result of the Biot splitting of $\boldsymbol\xi$ and $p$ in the thick domain as seen in \cite{BiotSplit}. The corresponding term arising from the Biot splitting in the thin domain contains $\left(h/H\right)^4$; the higher power is a result of the 4th order plate dynamics, and we see that this condition is more easily satisfied for smaller $H$. However, we can reduce the time-step condition to being linear in $h$ if the multiphysics problem satisfies certain parameter constraints, as we now show.

\begin{theorem}\label{thm:stab2}
Assume that the fluid-multilayered poroelastic system is isolated, that is, $\boldsymbol F_b = 0, G_b = 0, G_m = 0$, and $\boldsymbol F_f =  0$. Suppose the parameter conditions $\alpha^2 < c_0\lambda_b$ and $(\alpha^p)^2 < 12c_0^pD$ are satisfied and let $M = \min\left\{\frac{\rho_b}{C_{TI}^2},\frac{\rho_f}{A_{TI}^2}\right\}$. Then under the time-step condition
\begin{equation*}  \Delta t < \frac{M}{4}\cdot\min\left\{\frac{k_{min}}{C_{TR}^2}h,\kappa^p\left(\frac{h}{H}\right)\right\}, \end{equation*} 
the following {{stability}} estimate holds for some constants $\epsilon_{\boldsymbol\xi}, \epsilon_{\boldsymbol u}, \epsilon_{\boldsymbol u_b}, \epsilon_{u_m}, \epsilon_{p}, \epsilon_{q}, \epsilon_{\boldsymbol\eta},\epsilon_{\Lambda} \in (0,1)$:

\begin{align*} \footnotesize\begin{split} 
& \rho_b||\boldsymbol\xi_h^N||^2_{\Omega_b} + \epsilon_{\boldsymbol\eta}\lambda_b||\nabla\cdot\boldsymbol\eta_h^N||^2_{\Omega_b}+ 2\mu_b||\boldsymbol D(\boldsymbol \eta_h^N)||^2_{\Omega_b} +  \gamma||\boldsymbol\eta_h^N||^2_{\Omega_b}  + H\rho_p||v_h^N||^2_{\Gamma_+} + H\gamma^p||w_h^N||^2_{\Gamma_+}  \\
& + \epsilon_{\Lambda}H^3D||\Lambda_h^N||^2_{\Gamma_+}+  \epsilon_{p}c_0||p_h^0||^2_{\Omega_b} + \epsilon_{q}c_0^p||q_h^N||^2_{\Omega_p} + \rho_f||\boldsymbol u_h^N||^2_{\Omega_f} + \epsilon_{\boldsymbol\xi}\rho_b\sum_{n=0}^{N-1}||\boldsymbol\xi_h^{n+1} - \boldsymbol\xi_h^{n}||^2_{\Omega_b}  \\
& + \epsilon_{\boldsymbol\eta}\lambda_b\sum_{n=0}^{N-1}||\nabla\cdot(\boldsymbol\eta_h^{n}-\boldsymbol\eta_h^{n-1})||^2_{\Omega_b} + 2\mu_b\sum_{n=0}^{N-1}||\boldsymbol D(\boldsymbol\eta_h^{n+1}-\boldsymbol\eta_h^{n})||^2_{\Omega_b} +  \gamma\sum_{n=0}^{N-1}||\boldsymbol\eta_h^{n+1}-\boldsymbol\eta_h^{n}||^2_{\Omega_b} \\
&  + \rho_p\sum_{n=0}^{N-1}||v_h^{n+1} - v_h^n||^2_{\Gamma_+} + \epsilon_{\Lambda}H^3D\sum_{n=0}^{N-1}||\Lambda_h^{n} - \Lambda_h^{n-1}||^2_{\Gamma_+} + H\gamma^pc||w_h^{n+1}-w_h^n||^2_{\Gamma_+} \\
&  +  \epsilon_{p}c_0\sum_{n=0}^{N-1}||p_h^{n+1}-p_h^n||^2_{\Omega_b} + \epsilon_{q}c_0^p\sum_{n=0}^{N-1}||q_h^{n+1}-q_h^n||^2_{\Omega_p} + \epsilon_{\boldsymbol u_b}2\Delta t\sum_{n=0}^{N-1}||\boldsymbol\kappa^{-1/2}\boldsymbol u_{b,h}^{n+1}||^2_{\Omega_b} \\
& + \epsilon_{\boldsymbol u}\rho_f\sum_{n=0}^{N-1}||\boldsymbol u_h^{n+1}-\boldsymbol u_h^{n}||^2_{\Omega_f}  + \epsilon_{u_m}2\Delta t\sum_{n=0}^{N-1}||(\kappa^p)^{-1/2} u_{p,h}^{n+1}||^2_{\Omega_p}+ 4\Delta t \mu_f\sum_{n=0}^{N-1}||\boldsymbol D (\boldsymbol u_h^{n+1})||^2_{\Omega_f} \\
&  + 2\beta\Delta t\sum_{n=0}^{N-1} \left|\left|\boldsymbol P_{\boldsymbol\tau^p}\boldsymbol{u}_h^{n+1}\right|\right|^2_{\Gamma_-} \leq  \rho_b||\boldsymbol\xi_h^0||^2_{\Omega_b} + 2\lambda_b||\nabla\cdot\boldsymbol\eta_h^0||^2_{\Omega_b} + 2\mu_b||\boldsymbol D(\boldsymbol \eta_h^0)||^2_{\Omega_b} +  \gamma||\boldsymbol\eta_h^0||^2_{\Omega_b} \\
&  + H\rho_p||v_h^0||^2_{\Gamma_+}+ 2H^3D||\Lambda_h^0||^2_{\Gamma_+} + H\gamma^p||w_h^0||^2_{\Gamma_+} +  2c_0||p_h^0||^2_{\Omega_b} + 2c_0^p||q_h^0||^2_{\Omega_p} + \rho_f||\boldsymbol u_h^0||^2_{\Omega_f}. 
\end{split}  \end{align*}
\end{theorem}
\begin{proof}
The proof begins identically to that of \cref{thm:stab1}, except we leave the Biot splitting terms as they are and continue to bound the other two boundary terms arising from the splitting scheme in the same manner as before. This leads to the representation of the mixed and related terms in our inequality as a bilinear form on the variable
 \begin{equation*}  \vec x = \left(\begin{matrix} ||\boldsymbol\kappa^{-1/2}\boldsymbol u_{b,h}^{n+1}||_{\Omega_b} & ||\boldsymbol\xi_h^{n+1} - \boldsymbol\xi_h^{n}||_{\Omega_b}  & ||(\kappa^p)^{-1/2}u_{p,h}||_{\Omega_p} & \left|\left|\boldsymbol u_h^{n+1} - \boldsymbol u_h^{n}\right|\right|_{\Omega_f} \end{matrix}\right)^T \end{equation*} 
{{represented by}} the symmetric $4\times 4$ matrix
\begin{equation*}  \left(\begin{matrix}
2\Delta t & \frac{C_{TI}C_{TR}\Delta t\sqrt{2}}{\sqrt{hk_{min}}}  & 0  & \frac{A_{TI}C_{TR}\Delta t\sqrt{2}}{\sqrt{hk_{min}}} \\
& \rho_b &\frac{C_{TI}\sqrt{2H}\Delta t}{\sqrt{h\kappa^p}}  & 0 \\
 &  & 2\Delta t &  \frac{A_{TI}\sqrt{2H}\Delta t}{\sqrt{h\kappa^p}} \\ 
 &  &   & 0 \\
 &  &  & \rho_f
\end{matrix}\right). \end{equation*} 
Using \cref{lem:2} and bounding each of the terms appearing off the diagonal then yields
\begin{equation*} \small\begin{split} & \mathcal E(\boldsymbol A_h^{n+1}) + \epsilon_{\boldsymbol\xi}\rho_b||\boldsymbol\xi_h^{n+1} - \boldsymbol\xi_h^{n}||^2_{\Omega_b} + \lambda_b||\nabla\cdot(\boldsymbol\eta_h^{n+1}-\boldsymbol\eta_h^{n})||^2_{\Omega_b} + 2\mu_b||\boldsymbol D(\boldsymbol\eta_h^{n+1}-\boldsymbol\eta_h^{n})||^2_{\Omega_b} \\
& +  \gamma||\boldsymbol\eta_h^{n+1}-\boldsymbol\eta_h^{n}||^2_{\Omega_b}  + \rho_p||v_h^{n+1} - v_h^n||^2_{\Gamma_+} + H^3D||\Lambda_h^{n+1} - \Lambda_h^n||^2_{\Gamma_+} + H\gamma^p||w_h^{n+1}-w_h^n||^2_{\Gamma_+} \\
& +  c_0||p_h^{n+1}-p_h^n||^2_{\Omega_b} + c_0^p||q_h^{n+1}-q_h^n||^2_{\Omega_p} + \epsilon_{\boldsymbol u}\rho_f||\boldsymbol u_h^{n+1}-\boldsymbol u_h^{n}||^2_{\Omega_f} + \epsilon_{\boldsymbol u_b}2\Delta t||\boldsymbol\kappa^{-1/2}\boldsymbol u_{b,h}^{n+1}||^2_{\Omega_b} \\
&+ \epsilon_{u_p}2\Delta t||(\kappa^p)^{-1/2} u_{p,h}^{n+1}||^2_{\Omega_p} + 4\Delta t \mu_f||\boldsymbol D (\boldsymbol u_h^{n+1})||^2_{\Omega_f} + 2\beta\Delta t \left|\left|\boldsymbol P_{\boldsymbol\tau^p}\boldsymbol{u}_h^{n+1}\right|\right|^2_{\Gamma_-} \leq  \mathcal E(\boldsymbol A_h^{n}) \\
& + 2\alpha\Delta t(p_h^{n+1},\nabla\cdot(\boldsymbol\xi_h^{n+1} - \boldsymbol\xi_h^n))_{\Omega_b}  + 2H\alpha^p\Delta t(\nabla_{\boldsymbol\tau^p}\overline{(z + H/2)q_h^{n+1}},\nabla_{\boldsymbol\tau^p}( v_h^{n+1} - v_h^{n}))_{\Gamma_+}
\end{split}  \end{equation*} 
for some $\epsilon_{\boldsymbol\xi}, \epsilon_{\boldsymbol u}, \epsilon_{\boldsymbol u_b}, \epsilon_{u_p} \in (0,1)$ under the condition 
\begin{equation*}  \Delta t < \frac{1}{4}\min\left\{\frac{\rho_bk_{min}}{C_{TI}^2C_{TR}^2}h,\frac{\rho_fk_{min}}{A_{TI}^2C_{TR}^2}h,\frac{\rho_b\kappa^p}{C_{TI}^2}\left(\frac{h}{H}\right),\frac{\rho_f\kappa^p}{A_{TI}^2}\left(\frac{h}{H}\right)\right\}. \end{equation*} 
Letting $M = \min\left\{\frac{\rho_b}{C_{TI}^2},\frac{\rho_f}{A_{TI}^2}\right\}$ and factoring, this is precisely the same CFL condition as in the statement of the theorem. In order to bound the remaining coupling terms, we first write
\begin{equation*}\small \begin{split} 
& \Delta t(p_h^{n+1},\nabla\cdot(\boldsymbol\xi_h^{n+1} - \boldsymbol\xi_h^n))_{\Omega_b} \\
& = \left(p_h^{n+1},\nabla\cdot\left[(\boldsymbol\eta_h^{n+1} - \boldsymbol\eta_h^n) - (\boldsymbol\eta_h^{n} - \boldsymbol\eta_h^{n-1})\right]\right)_{\Omega_b} \\
& = (p_h^{n+1},\nabla\cdot(\boldsymbol\eta_h^{n+1} - \boldsymbol\eta_h^n))_{\Omega_b} - (p_h^{n},\nabla\cdot(\boldsymbol\eta_h^{n} - \boldsymbol\eta_h^{n-1}))_{\Omega_b} - (p_h^{n+1} - p_h^n,\nabla\cdot(\boldsymbol\eta_h^{n} - \boldsymbol\eta_h^{n-1}))_{\Omega_b}
\end{split} \end{equation*}
and similarly 
\begin{equation*}\small \begin{split}  \Delta t &(\nabla_{\boldsymbol\tau^p}\overline{(z + H/2)q_h^{n+1}},\nabla_{\boldsymbol\tau^p}( v_h^{n+1} - v_h^{n}))_{\Gamma_+} \\
 =& -(\overline{(z + H/2)q_h^{n+1}},(\Lambda_h^{n+1} - \Lambda_h^{n})-(\Lambda_h^{n} - \Lambda_h^{n-1}))_{\Gamma_+} \\
 =&-(\overline{(z + H/2)q_h^{n+1}},\Lambda_h^{n+1} - \Lambda_h^{n})_{\Gamma_+} \\
& + (\overline{(z + H/2)q_h^{n}},\Lambda_h^{n} - \Lambda_h^{n-1})_{\Gamma_+} + (\overline{(z + H/2)(q_h^{n+1} - q_h^{n})},\Lambda_h^{n} - \Lambda_h^{n-1})_{\Gamma_+}.
\end{split} \end{equation*}
Letting 
\begin{equation*} \begin{split} \mathcal C(\boldsymbol A_h^{n+1}) 
& = \lambda_b||\nabla\cdot(\boldsymbol\eta_h^{n+1}-\boldsymbol\eta_h^{n})||^2_{\Omega_b} + H^3D||\Lambda_h^{n+1}- \Lambda_h^n||^2_{\Gamma_+}  \\
& - 2\alpha(p_h^{n+1},\nabla\cdot(\boldsymbol\eta_h^{n+1} - \boldsymbol\eta_h^n))_{\Omega_b} + 2H\alpha^p(\overline{(z + H/2)q_h^{n+1}},\Lambda_h^{n+1} - \Lambda_h^{n})_{\Gamma_+},
\end{split} \end{equation*}
we arrive at the estimate
\begin{equation*}\begin{split} & \mathcal E(\boldsymbol A_h^{n+1}) +  \mathcal C(\boldsymbol A_h^{n+1}) + \epsilon_{\boldsymbol\xi}\rho_b||\boldsymbol\xi_h^{n+1} - \boldsymbol\xi_h^{n}||^2_{\Omega_b} + \lambda_b||\nabla\cdot(\boldsymbol\eta_h^{n}-\boldsymbol\eta_h^{n-1})||^2_{\Omega_b} \\
&  + 2\mu_b||\boldsymbol D(\boldsymbol\eta_h^{n+1}-\boldsymbol\eta_h^{n})||^2_{\Omega_b} +  \gamma||\boldsymbol\eta_h^{n+1}-\boldsymbol\eta_h^{n}||^2_{\Omega_b}  + \rho_p||v_h^{n+1} - v_h^n||^2_{\Gamma_+} +  c_0||p_h^{n+1}-p_h^n||^2_{\Omega_b}    \\
&+ H^3D||\Lambda_h^{n} - \Lambda_h^{n-1}||^2_{\Gamma_+}+ H\gamma^p||w_h^{n+1}-w_h^n||^2_{\Gamma_+} + c_0^p||q_h^{n+1}-q_h^n||^2_{\Omega_p} \\
&+ \epsilon_{\boldsymbol u}\rho_f||\boldsymbol u_h^{n+1}-\boldsymbol u_h^{n}||^2_{\Omega_f} + \epsilon_{\boldsymbol u_b}2\Delta t||\boldsymbol\kappa^{-1/2}\boldsymbol u_{b,h}^{n+1}||^2_{\Omega_b}  + \epsilon_{u_p}2\Delta t||(\kappa^p)^{-1/2} u_{p,h}^{n+1}||^2_{\Omega_p}  \\
& + 4\Delta t \mu_f||\boldsymbol D (\boldsymbol u_h^{n+1})||^2_{\Omega_f} + 2\beta\Delta t \left|\left|\boldsymbol P_{\boldsymbol\tau^p}\boldsymbol{u}_h^{n+1}\right|\right|^2_{\Gamma_-} \leq  \mathcal E(\boldsymbol A_h^{n}) + \mathcal C(\boldsymbol A_h^{n}) \\
& + 2H\alpha^p(\overline{(z + H/2)(q_h^{n+1} - q_h^{n})},\Lambda_h^{n} - \Lambda_h^{n-1})_{\Gamma_+}- 2\alpha(p_h^{n+1} - p_h^n,\nabla\cdot(\boldsymbol\eta_h^{n} - \boldsymbol\eta_h^{n-1}))_{\Omega_b}  .
\end{split}  \end{equation*}
Using Cauchy–Schwarz, the first coupling term is bounded by
\begin{equation*}  \left|2\alpha(p_h^{n+1} - p_h^n,\nabla\cdot(\boldsymbol\eta_h^{n} - \boldsymbol\eta_h^{n-1}))_{\Omega_b}\right| \leq 2\alpha||p_h^{n+1}-p_h^n||_{\Omega_b}||\nabla\cdot(\boldsymbol\eta_h^{n}-\boldsymbol\eta_h^{n-1})||_{\Omega_b}  \end{equation*} 
where $||p_h^{n+1}-p_h^n||_{\Omega_b}$ appears on the left-hand side with coefficient $c_0$ and $||\nabla\cdot(\boldsymbol\eta_h^{n}-\boldsymbol\eta_h^{n-1})||_{\Omega_b}$ with coefficient $\lambda_b$. Thus by \cref{lem:1}, the coupling term can absorbed into the left side if  $\alpha^2 < c_0\lambda$, which becomes our first parameter condition. To similarly handle the second coupling term, we first use the bound
\begin{equation*} \small \left|H\overline{\left(z+H/2\right)q}\right|^2 = \left|\int_{-H}^0\left(z + \frac{H}{2}\right)q \ dz \right|^2 \leq\int_{-H}^0\left(z + \frac{H}{2}\right)^2 \ dz\int_{-H}^0q^2 \ dz = \frac{H^3}{12}\int_{-H}^0q^2 \ dz  \end{equation*}
which after integrating over $x$ and $y$ yields $H\left|\left|\overline{\left(z+H/2\right)q}\right|\right|_{\Gamma_+} \leq \frac{H^{3/2}}{\sqrt{12}}||q||_{\Omega_p}$. This combined with Cauchy–Schwarz gives
 \begin{equation*} \small \left|2H\alpha^p(\overline{(z + H/2)(q_h^{n+1} - q_h^{n})},\Lambda_h^{n} - \Lambda_h^{n-1})_{\Gamma_+}\right| \leq \frac{2\alpha H^{3/2}}{\sqrt{12}}||q_h^{n+1}-q_h^n||_{\Omega_p}||\Lambda_h^{n} - \Lambda_h^{n-1}||_{\Gamma_+}  \end{equation*} 
where $||q_h^{n+1}-q_h^n||_{\Omega_p}$ has coefficient $c_0^p$ in the estimate and $||\Lambda_h^{n} - \Lambda_h^{n-1}||_{\Gamma_+}$ has coefficient $H^3D$. Again by \cref{lem:1}, the coupling term is absorbed under the parameter condition $\frac{H^3(\alpha^p)^2}{12} < H^3c_0^pD$. If we collect all the resulting terms on the left-hand side besides $\mathcal E(\boldsymbol A_h^{n+1})$ and $\mathcal C(\boldsymbol A_h^{n+1})$ in a total dissipation term $\mathcal N(\boldsymbol A_h^{n+1})$, summing over $0 \leq n \leq N-1$ gives
\begin{equation*}   \mathcal E(\boldsymbol A_h^{N}) + \mathcal C(\boldsymbol A_h^{N}) + \sum_{n=0}^{N-1}\mathcal N(\boldsymbol A_h^{n+1}) \leq \mathcal E(\boldsymbol A_h^{0}) + \mathcal C(\boldsymbol A_h^{0})   \end{equation*}
and so the stated energy inequality follows assuming that $E(\boldsymbol A_h^{N}) + \mathcal C(\boldsymbol A_h^{N}) \geq 0$ for every $N$. However, showing that
\begin{equation*}  \left|2\alpha(p_h^{n+1},\nabla\cdot(\boldsymbol\eta_h^{n+1} - \boldsymbol\eta_h^{n}))_{\Omega_b}\right| \leq 2\alpha||p_h^{n+1}||_{\Omega_b}||\nabla\cdot(\boldsymbol\eta_h^{n+1}-\boldsymbol\eta_h^{n})||_{\Omega_b},  \end{equation*} 
 \begin{equation*}  \left|2H\alpha^p(\overline{(z + H/2)q_h^{n+1}},\Lambda_h^{n+1} - \Lambda_h^{n})_{\Gamma_+}\right| \leq \frac{2\alpha H^{3/2}}{\sqrt{12}}||q_h^{n+1}||_{\Omega_p}||\Lambda_h^{n+1} - \Lambda_h^{n}||_{\Gamma_+}  \end{equation*} 
follows an identical strategy to the one used above and gives $E(\boldsymbol A_h^{N}) + \mathcal C(\boldsymbol A_h^{N}) \geq 0$ under the exact same parameter conditions specified in the statement of the theorem.
\end{proof}

\section{Numerical Simulations} \label{Simulations}

In this section we present numerical examples demonstrating the performance of the scheme in different two dimensional settings. We first test the algorithm on a benchmark problem where both $\Omega_b$ and $\Omega_f$ are unit squares in order to investigate convergence properties and the effect of the plate thickness $H$ in the dynamics of the fluid-structure model. We then apply the scheme to a physiologically relevant FPSI setting related to blood flow through a vessel.

We discretize in space by using piecewise linear $\mathbb P_1$ elements everywhere except in the fluid problem where we instead use Taylor-Hood ($\mathbb P_2 -\mathbb P_1$) elements. For errors, if $\chi$ is a variable of interest over a domain $\mathcal O$ which is numerically approximated by $\hat \chi$, we consider the relative error $e_\chi = \frac{||\chi - \hat\chi||_{\mathcal O}}{||\chi||_{\mathcal O}}$. The scheme was implemented using the FEniCS finite element library in Python (see \cite{fenics, fenicsbook}). We implement the averaging operator $\overline{f} = \frac{1}{H}\int_{-H/2}^{H/2}f \ dz$ by using Simpson's rule on a structured mesh; however, more sophisticated implementations may be necessary in the case of more complicated meshes or domains. For the implementation of the extension operator $\tilde f$ where $\left.\tilde f\right|_{\Gamma_-} = f$ and $\nabla_{\boldsymbol n^p}\tilde f = 0$, we solve the Laplace equation $\Delta_{\boldsymbol n^p}\tilde f = 0$ with the Dirichlet boundary condition $\left.\tilde f\right|_{\Gamma_-} = f$ on $\Gamma_-$ and the Neumann boundary condition $\left.\nabla_{\boldsymbol n^p}\tilde f\right|_{\Gamma_+}  = 0$ on $\Gamma_+$.

\subsection{Manufactured Solution}

For the first test case, we consider a two - dimensional manufactured solution to the coupled problem that satisfies the complete set of coupling conditions outlined previously (apart from balance of forces \textbf{(D1)} which includes a forcing term). We construct this solution on the computational domain $\Omega = [0,1]\times[-2,1]$ so that each of the full dimensional subdomains $\Omega_b  = [0,1]\times[0,1]$, $\Omega_p  = [0,1]\times[-1,0]$, and $\Omega_f  = [0,1]\times[-2,-1]$ are unit squares coupled on the one-dimensional boundaries $\Gamma_+ = [0,1]\times\{0\}$ and $\Gamma_- = [0,1]\times\{-1\}$. Due to the primal mixed formulation of the plate structure equation, it is also required that the manufactured solution variable $w = \left.\boldsymbol\eta\right|_{\Gamma_+}$ satisfies the homogenous Neumann condition $\partial_x w(0) =\partial_x w(1) = 0$.  Solution functions that satisfy all the specified boundary conditions are given by

\begin{equation*}  \begin{cases}
p  = 2\pi y\cos(2\pi x)\cos(y)(t + 1) & \text{ on } \Omega_b,  \\
q = \cos(2\pi x)\sin(2\pi y)(t + 1) & \text{ on } \Omega_m, \\
\boldsymbol \eta  = \left(\begin{matrix}\sin(2\pi x)\sin(2\pi y)(t+1) \\ 2\pi \cos(2\pi x)\cos(2\pi y)\left(\frac{t^2}{2}+t+1\right) \end{matrix}\right) & \text{ on } \Omega_b,   \\
\boldsymbol u = \left(\begin{matrix}\cos(x)\exp(-(y+1))(t+1) \\ \sin(x)\left(1-\exp(-(y+1))\right)(t+1)\end{matrix}\right) & \text{ on } \Omega_f, \\
\pi  = 2\sin(x)\cos(2\pi y)(t + 1) & \text{ on } \Omega_f, \\
\end{cases} \end{equation*} 
where the remaining variables $\boldsymbol u_b, u_p, \boldsymbol\xi, w, v$, and $\Lambda$ are uniquely determined from these functions and are all nonzero. Forcing terms $\boldsymbol F_b, G_b, F_p, G_p, \boldsymbol F_f$ are also calculated from the given solution functions using unit values for all constants except $\gamma, \gamma^p,$ and $\gamma_{pen}$, which are set to zero. Exact Dirichlet conditions are imposed along the boundary wherever appropriate, and exact initial conditions are used for $\boldsymbol\eta, \boldsymbol\xi, p, q$, and $\boldsymbol u$.

\begin{figure}[t]
\begin{center}
\includegraphics[scale=.13]{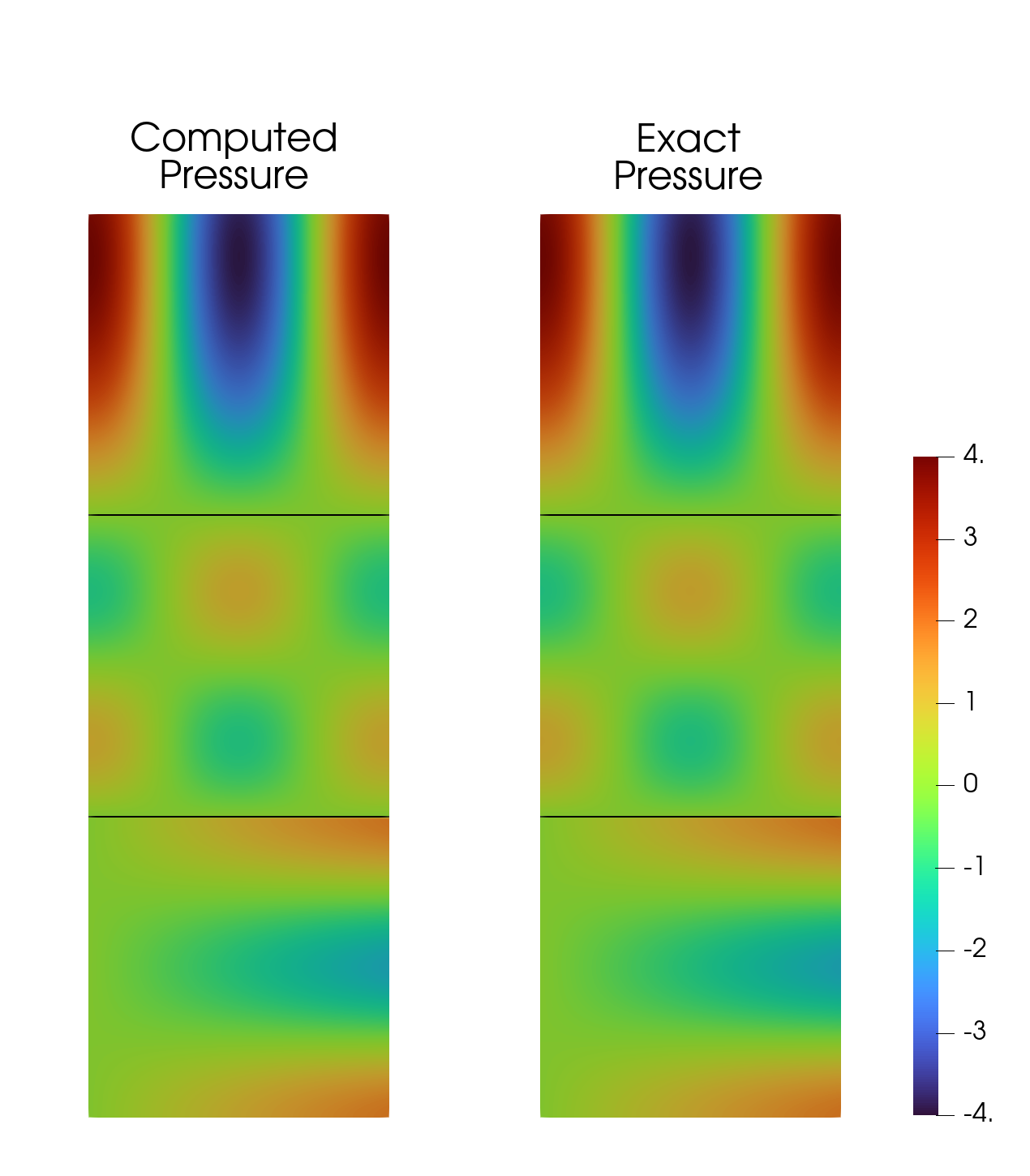}
\includegraphics[scale=.13]{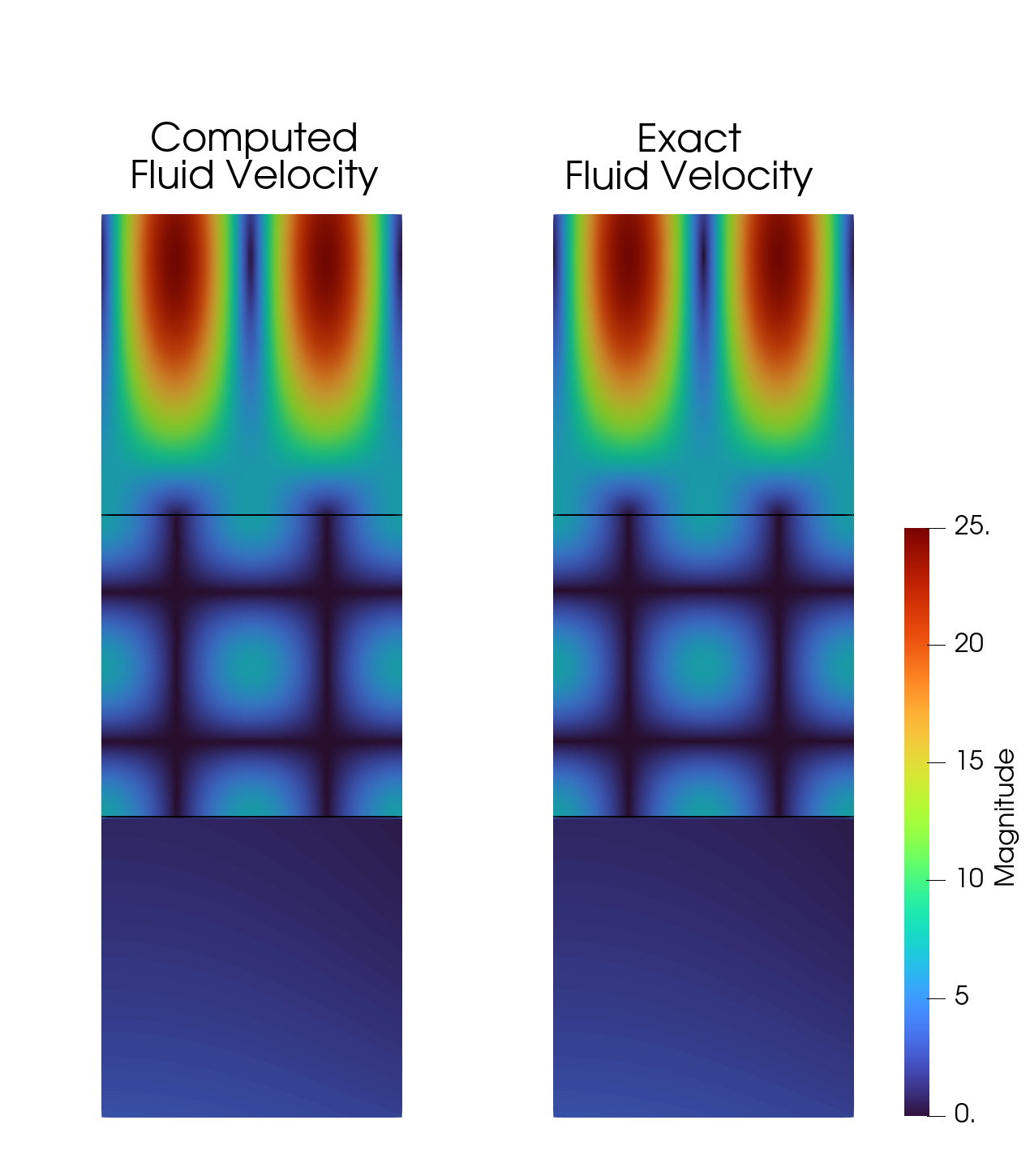}
\includegraphics[scale=.12]{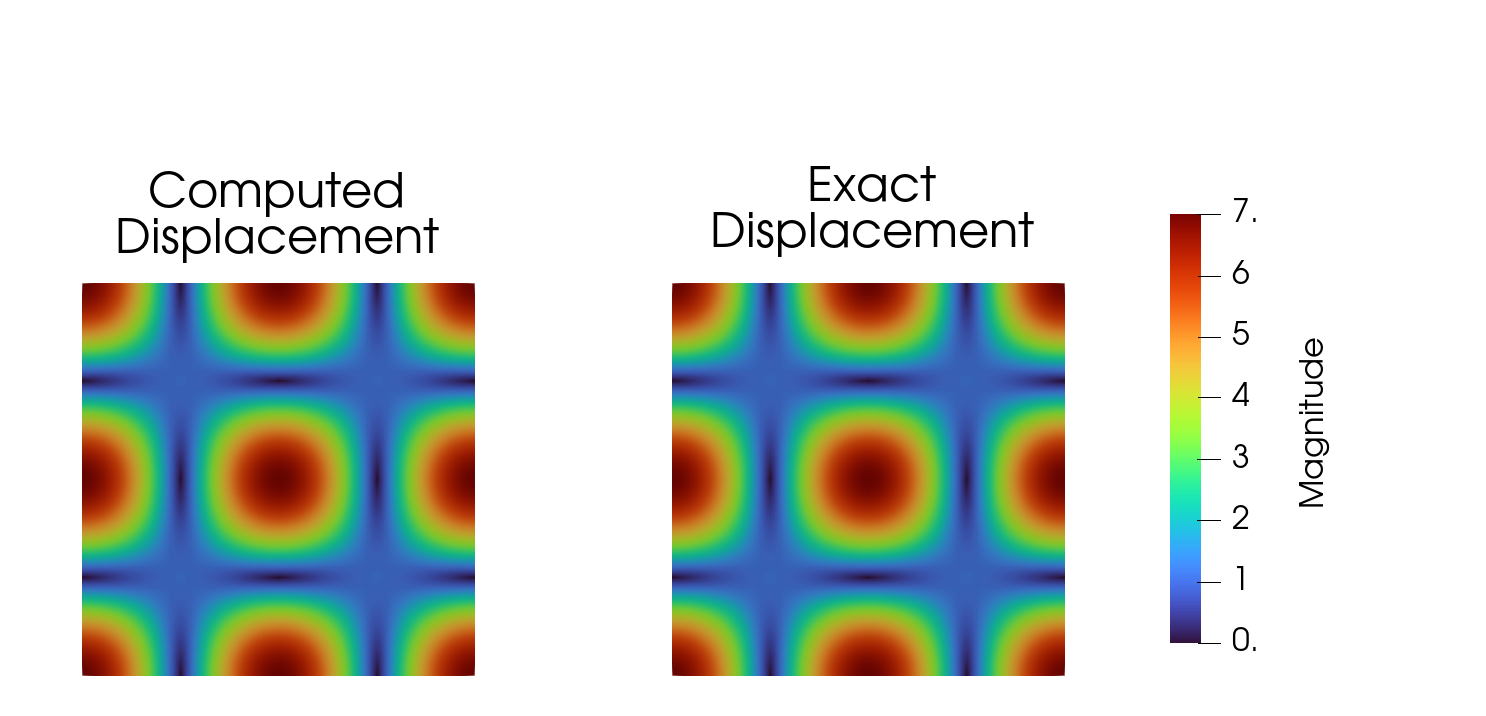}
\end{center}
\caption{Visualization of computed and exact pressures (top left two panels), computed and exact fluid velocities (top right two panels), and computed and exact structure displacements in the thick Biot domain (bottom two panels). Black lines indicate the interfaces $\Gamma_+$ and $\Gamma_-$}\label{fig:manvis}
\end{figure}

A visualization (\cref{fig:manvis}) was first created by solving the problem with the calculated forcing terms until time $T = 0.1$ with a step size of $\Delta t = 0.001$ and a spatial mesh size of $h = 120^{-1}$ and then comparing to the exact solution (\cref{tab:errs}). 
\begin{table}[h]
    \centering
    \label{tab:errs}
	\vspace{2mm}
   \begin{tblr}{colspec = {c|c|c|c|c|c|c|c|c}, rowsep=3pt}
     $e_p$ & $e_q$ & $e_{\boldsymbol\eta}$ & $e_{\boldsymbol\xi}$ & $e_w$ & $e_v$ & $e_{\boldsymbol u}$ & $e_\pi$ \\ \hline
   \small $1.6e-3$ & \small $2.2e-2$ &\small $1.9e-3$ & \small$8.1e-3$ &\small $5.5e-3$ &\small $3.7e-2$  &\small $1.5e-4$ &\small $7.2e-3$
\end{tblr}
\caption{Relative error between exact and numerical solutions}
\end{table}

Spatial convergence was then tested by calculating the relative errors for the same time interval and step size but with spatial mesh sizes $h = 40^{-1}, 60^{-1}$, and $80^{-1}$ (\cref{fig:sconv}), while convergence in time was tested by solving until time $T = 1$ by using a fixed spatial mesh size of $h = 200^{-1}$ and step sizes $\Delta t = 0.01, 0.02,$ and $0.04$ (\cref{fig:tconv}). Second-order convergence in space and first-order convergence in time is observed for all physical variables.

To test long term stability, the solution was calculated until time $T = 10$ with a time step of $\Delta t = 0.1$ and a mesh size of $h = 120^{-1}$. Energy of the numerical solution and the exact solution are calculated at each time step and plotted over time (\cref{fig:longterm}). An excellent match is observed.

\begin{figure}[h]
\centering
\includegraphics[scale=.4]{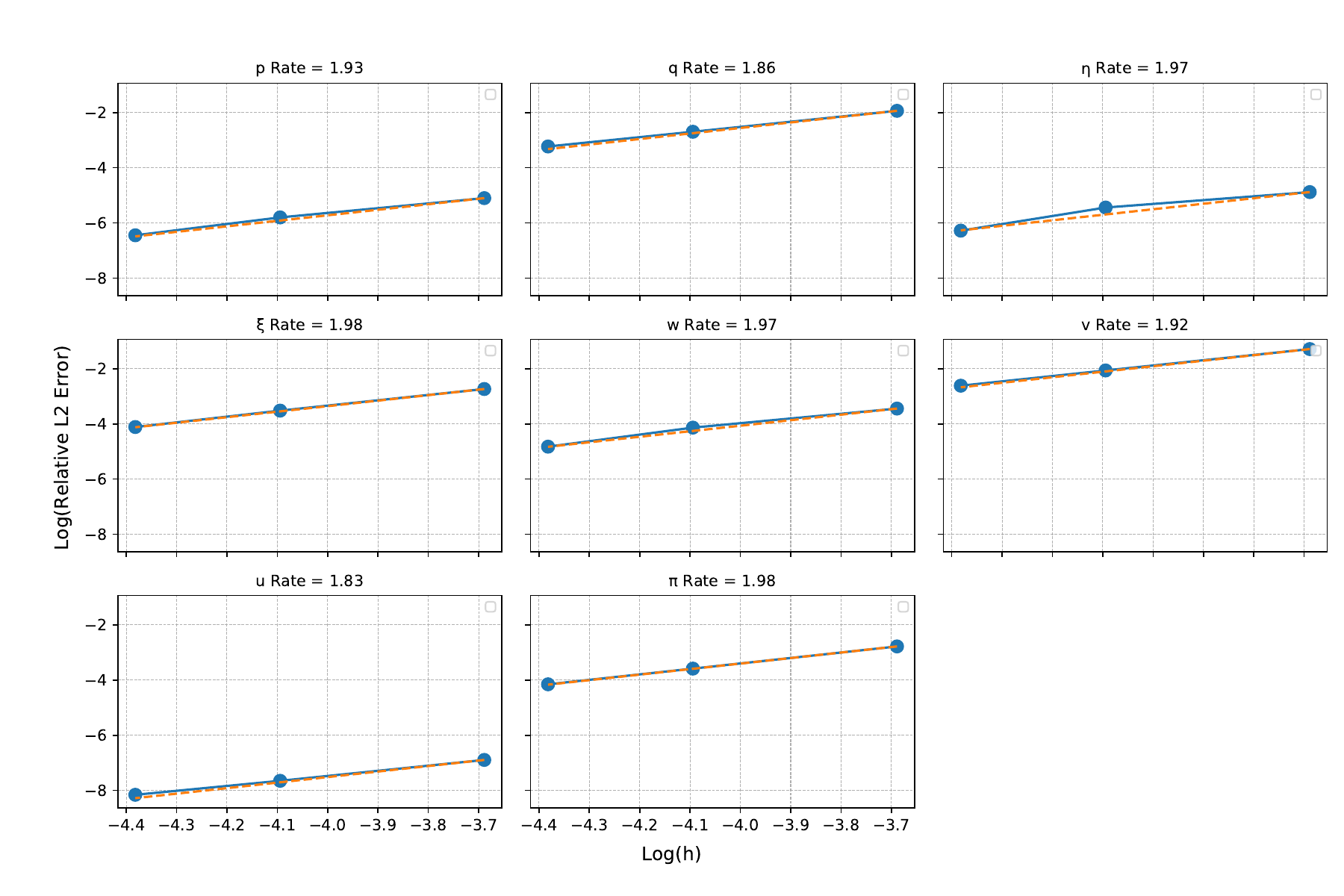}
\caption{Space convergence (blue line) plotted as a log-log plot against a line of slope 2 (dotted yellow line).}\label{fig:sconv}
\end{figure}

\begin{figure}[h]
\centering
\includegraphics[scale=.4]{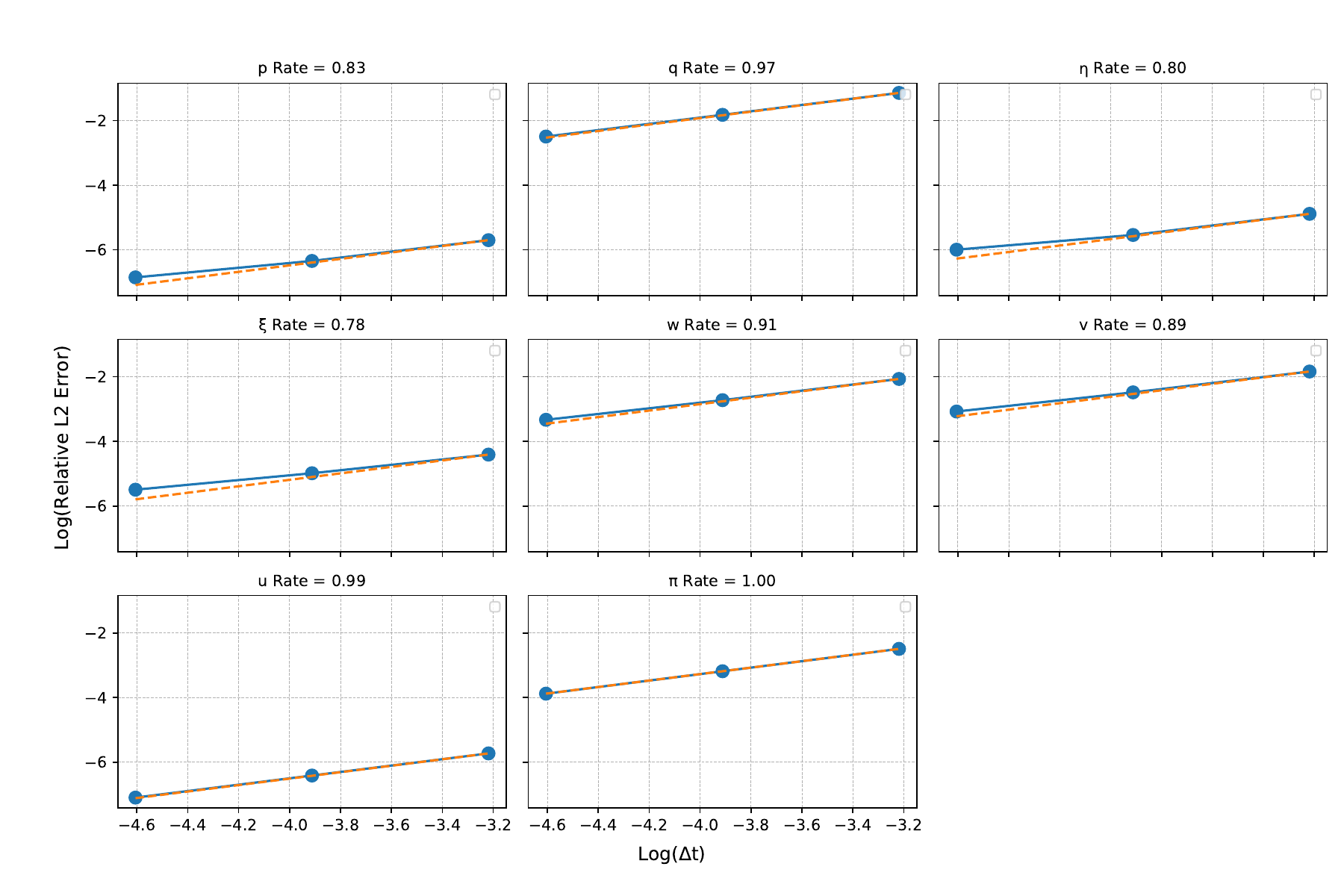}
\caption{Time convergence (blue line) plotted as a log-log plot against a line of slope 1 (dotted yellow line). }\label{fig:tconv}
\end{figure}

\begin{figure}[h]
\centering
\includegraphics[scale=.5]{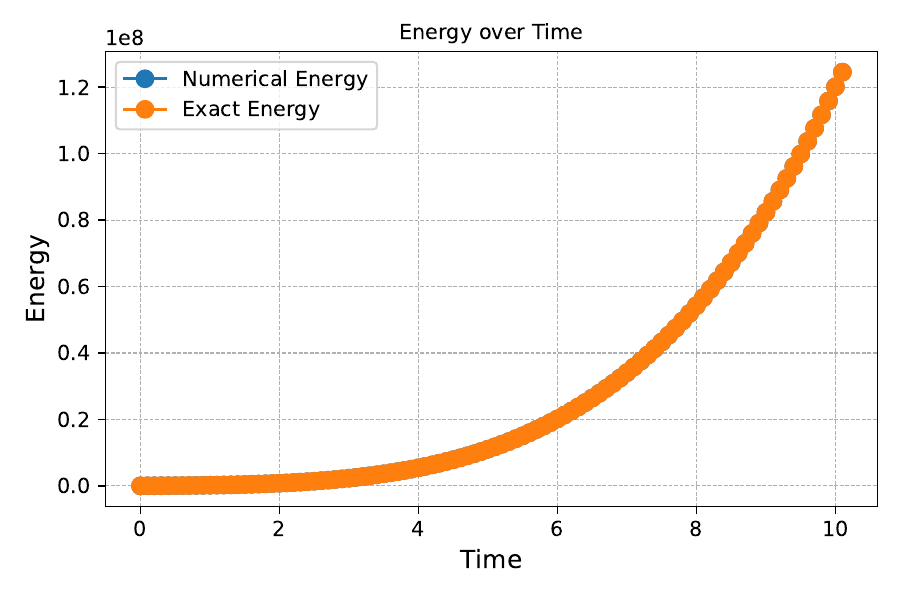}
\caption{Numerical energy (blue line) closely matches the exact energy of the manufactured solution (orange line) when plotted over $T = 10$ seconds. The two lines are superimposed.}\label{fig:longterm}
\end{figure}

\subsection{Vessel Flow Problem}

\tikzmath{\RS = .5; \LL =8; \HH = .1;\RF = 1.5;} 

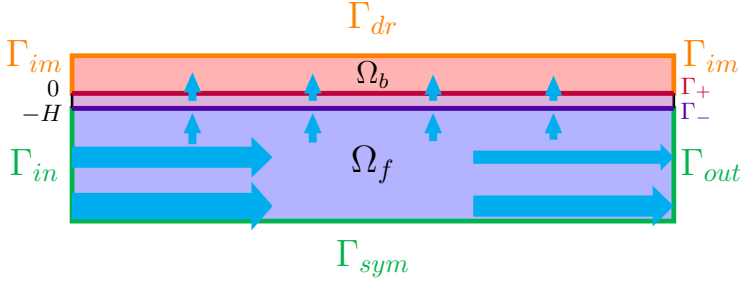
\begin{figure}[h]
\centering
\begin{tikzpicture}[fill opacity=1,text opacity=1]
\fill[line width=1.8pt,blue,fill opacity=.3] (0,-2*\HH-\RF) rectangle (\LL,-2*\HH) node[midway] {\Large\color{black}$\Omega_f$};
\draw[line width=1.8pt,green!70!blue!] (0,-2*\HH-\RF)  node[shift={(-0.5,\RF/2)}] {\Large $ \Gamma_{in}$} node[shift={(\LL/2,-0.5)}] {\Large $ \Gamma_{sym}$} rectangle (\LL,-2*\HH) node[shift={(0.5,-\RF/2)}] {\Large $ \Gamma_{out}$};
\fill[line width=1.8pt, red,fill opacity=.3] (0,0)  rectangle (\LL,\RS) node[midway] {\large\color{black}$\Omega_b$};
\draw[line width=1.8pt, orange] (0,0) node[shift={(-0.5,\RS)}] {\Large $ \Gamma_{im}$}  rectangle (\LL,\RS) node[shift={(0.5,0)}] {\Large $ \Gamma_{im}$} node[shift={(-\LL/2,0.5)}] {\Large $ \Gamma_{dr}$};
\draw[line width=1pt] (\LL,-2*\HH)  -- (\LL,0);
\draw[line width=1pt] (0,-2*\HH)  -- (0,0);
\draw[line width=1.8pt,red!80!blue!] (0,0)  node[shift={(-0.25,.07)}] {\small \color{black} $0$} -- (\LL,0) node[shift={(.3,.07)}] {\small $ \Gamma_+$};
\draw[line width=1.8pt,red!30!blue!] (0,-2*\HH)  node[shift={(-0.4,-.07)}] {\small\color{black} $-H$} -- (\LL,-2*\HH) node[shift={(.3,-.07)}] { \small$ \Gamma_-$};
\fill[line width=1.8pt,red!50!blue!,fill opacity=.3] (0,-2*\HH) rectangle (\LL,0);
\draw[-{Triangle[width=15pt,length=8pt,cyan]}, line width=8pt,cyan](0,-\HH-\RF/2) -- (\LL/3, -\HH-\RF/2);
\draw[-{Triangle[width=10pt,length=5pt,cyan]}, line width=5pt,cyan](2*\LL/3,-\HH-\RF/2) -- (\LL, -\HH-\RF/2);
\draw[-{Triangle[width=18pt,length=8pt,cyan]}, line width=10pt,cyan](0,-\HH-\RF+.1) -- (\LL/3, -\HH-\RF+.1);
\draw[-{Triangle[width=15pt,length=8pt,cyan]}, line width=8pt,cyan](2*\LL/3,-\HH-\RF+.1) -- (\LL, -\HH-\RF+.1);
\draw[-{Triangle[width=6pt,length=8pt,cyan]}, line width=3pt,cyan](\LL/5,-\HH-\RF/2.6) -- (\LL/5, -\HH-\RF/10);
\draw[-{Triangle[width=6pt,length=8pt,cyan]}, line width=3pt,cyan](2*\LL/5,-\HH-\RF/2.7) -- (2*\LL/5, -\HH-\RF/10);
\draw[-{Triangle[width=6pt,length=8pt,cyan]}, line width=3pt,cyan](3*\LL/5,-\HH-\RF/2.8) -- (3*\LL/5, -\HH-\RF/10);
\draw[-{Triangle[width=6pt,length=8pt,cyan]}, line width=3pt,cyan](4*\LL/5,-\HH-\RF/2.9) -- (4*\LL/5, -\HH-\RF/10);
\draw[-{Triangle[width=6pt,length=8pt,cyan]}, line width=3pt,cyan](\LL/5,-\HH) -- (\LL/5, \RS/1.7);
\draw[-{Triangle[width=6pt,length=8pt,cyan]}, line width=3pt,cyan](2*\LL/5,-\HH) -- (2*\LL/5, \RS/1.8);
\draw[-{Triangle[width=6pt,length=8pt,cyan]}, line width=3pt,cyan](3*\LL/5,-\HH) -- (3*\LL/5, \RS/1.9);
\draw[-{Triangle[width=6pt,length=8pt,cyan]}, line width=3pt,cyan](4*\LL/5,-\HH) -- (4*\LL/5, \RS/2);
\end{tikzpicture}
\caption{Computational domain for the pressure-driven flow with physical boundary conditions}
\label{fig:UnitFlowDomain}
\end{figure}

In this example we investigate the performance of our numerical method using physical boundary conditions on a physiologically motivated vessel domain. We will use this example to examine the solution as the thickness of the plate, $H$, goes to zero, and how it compares to the solution obtained using a (single layered) Stokes-Biot coupled problem. The solution to the Stokes-Biot problem will be calculated using a different solver, namely, we will use the monolithic scheme outlined in \cite{MassPen}. The computational domain is $\Omega = [0,L]\times[-H-R_f,R_b]$ so that $\Omega_b  = [0,L]\times[0,R_b]$ and $\Omega_f  = [0,L]\times[-H-R_f,-H]$ are rectangular domains representing the poroelastic media-adventitia layer of arterial walls and the artery lumen, respectively. These are coupled to the thin plate $\Omega_p  = [0,L]\times[-H,0]$ on the one-dimensional boundaries $\Gamma_+ = [0,L]\times\{0\}$ and $\Gamma_- = [0,L]\times\{-H\}$ representing a thin poroelastic intimal layer of arterial walls. In order to create a pressure-driven flow in the coupled system, we partition the external boundary $\Gamma_b$ of the thick Biot domain into a drained boundary $\Gamma_{dr}$ and an impermeable boundary $\Gamma_{im}$ and partition the external boundary $\Gamma_f$ of the fluid domain into an inflow boundary $\Gamma_{in}$, an outflow boundary $\Gamma_{out}$, and a symmetric boundary $\Gamma_{sym}$ (see \cref{fig:UnitFlowDomain}). The boundary conditions are given by
\begin{equation*} 
\begin{cases}
\begin{array}{lll}
\sigma_b(\boldsymbol\eta,p)\boldsymbol n = 0, & p = 0 & \text{ on } \Gamma_{dr}, \\
\boldsymbol\eta = 0, & \boldsymbol u_b\cdot\boldsymbol n = 0 & \text{ on } \Gamma_{im}, \\
\sigma_f(\boldsymbol u,\pi)\boldsymbol n = -P_{in}(t)\boldsymbol n & & \text{ on } \Gamma_{in}, \\
\sigma_f(\boldsymbol u,\pi)\boldsymbol n = 0 &  & \text{ on } \Gamma_{out}, \\
\sigma_f(\boldsymbol u,\pi)\boldsymbol n\cdot \boldsymbol \tau = 0, & \boldsymbol u\cdot\boldsymbol n = 0 & \text{ on } \Gamma_{sym}
\end{array}
\end{cases}
\end{equation*} 
where $P_{in}(t) = \frac{1}{2}P_{max}\left(1 - \cos\left(2\pi t/T_{pulse}\right)\right)1_{t < T_{pulse}}(t)$ is the inlet pressure given as a pulse function with maximum pressure of $P_{max}$ on the time interval $[0,T_{pulse}]$. For this simulation we set all the forcing terms to zero and choose physiological parameters for the constants as specified in \cref{tab:param}. In addition, we choose the plate elasticity coefficient $D = \frac{4\mu_b(\lambda_b + \mu_b)}{3(\lambda_b + 2\mu_b)}$ (dyne/cm$^2$) and use the pulse parameters $P_{max} = 13333$ dyne/cm$^2$, $T_{pulse} = 0.003$ s.

\begin{table}[h]
    \centering
    \begin{tabular}{|l l|l l|}
        \hline
        \footnotesize Parameters &\footnotesize Values &\footnotesize Parameters &\footnotesize Values \\
        \hline
        \footnotesize Fluid radius $R_f$ (cm) &\footnotesize 0.5 &\footnotesize Length $L$ (cm) &\footnotesize 5 \\
        \footnotesize PE wall thickness $R_b$ (cm) &\footnotesize 0.1  & \footnotesize Plate thickness $H$ (cm) &\footnotesize variable  \\
        \footnotesize Fluid density $\rho_f$ (g/cm$^3$) &\footnotesize 1  & \footnotesize PE wall density $\rho_b$ (g/cm$^3$) &\footnotesize 1.1  \\
        \footnotesize Membrane density $\rho_p$ (g/cm$^3$) &\footnotesize 1.1 &\footnotesize Fluid viscosity $\mu_f$ (g/cm s) &\footnotesize  0.035 \\
        \footnotesize Lamé coeff. $\mu_b$ (dyne/cm$^2$) &\footnotesize $5.58 \times 10^5$ &\footnotesize Lamé coeff. $\lambda_b$ (dyne/cm$^2$) &\footnotesize $1.7 \times 10^6$ \\
        \footnotesize Permeabilities $\kappa, \kappa^p$ (cm$^3$ s/g) & \footnotesize $1 \times 10^{-8}$ & \footnotesize Storage coeff. $c_0, c_0^p$ (cm$^2$/dyne) &\footnotesize $1 \times 10^{-3}$ \\
       \footnotesize Biot--Willis constants $\alpha, \alpha^p$ & \footnotesize 1 &\footnotesize Spring coeff. $\gamma, \gamma^p$ (dyne/cm$^4$) &\footnotesize $4 \times 10^6$ \\
\footnotesize Slip coeff. $\beta$ & \footnotesize 1 &\footnotesize Penalty coeff. $\gamma_{pen}$ &\footnotesize $1.2 \times 10^3$ \\
        \hline
    \end{tabular}
\caption{Geometry, fluid, and structure parameters used in vessel problem}
\label{tab:param}
\end{table}

A triangulation was created based on four meshing parameters; the horizontal width $h_L = L/300$, the vertical width $h_f = R_f/25$ in the fluid domain, the vertical width $h_b = R_b/4$ in the thick Biot domain, and the vertical width $h_p = H/3$ in the thin Biot domain. Note that the resulting triangles therefore have smaller minimum angles in the Biot domains and that $h_L$ is the relevant mesh scale for use in the penalty term (which is an integral over the horizontal lines $\Gamma_{\pm}$). The step size was set to $\Delta t = 0.0005$ s and the simulation was ran until time $T = 0.014$ seconds. In order to validate the solver, we also simulate the corresponding Stokes-Biot problem without the plate on the domain $[0,L]\times[-R_f,R_b]$ with the exact same boundary conditions using the dual-mixed formulation monolithic scheme outlined in \cite{MassPen}. We see that, as the thickness of the plate $H$ goes to $0$, the dynamics of the FPSI problem with a plate converges to the no-plate SB problem (\cref{fig:Hconv,fig:velFlow}). Furthermore, the presence of the plate with mass and poroelastic energy at the interface regularizes the dynamics as can be seen by comparing the number and magnitude of the oscillations in the displacement of the interface at time $t = 0.0075$ s, shown in \cref{fig:Hconv}, and the number of oscillations in the fluid pressure shown in \cref{fig:velFlow}.

\begin{figure}[h]
\centering
\includegraphics[scale=.28]{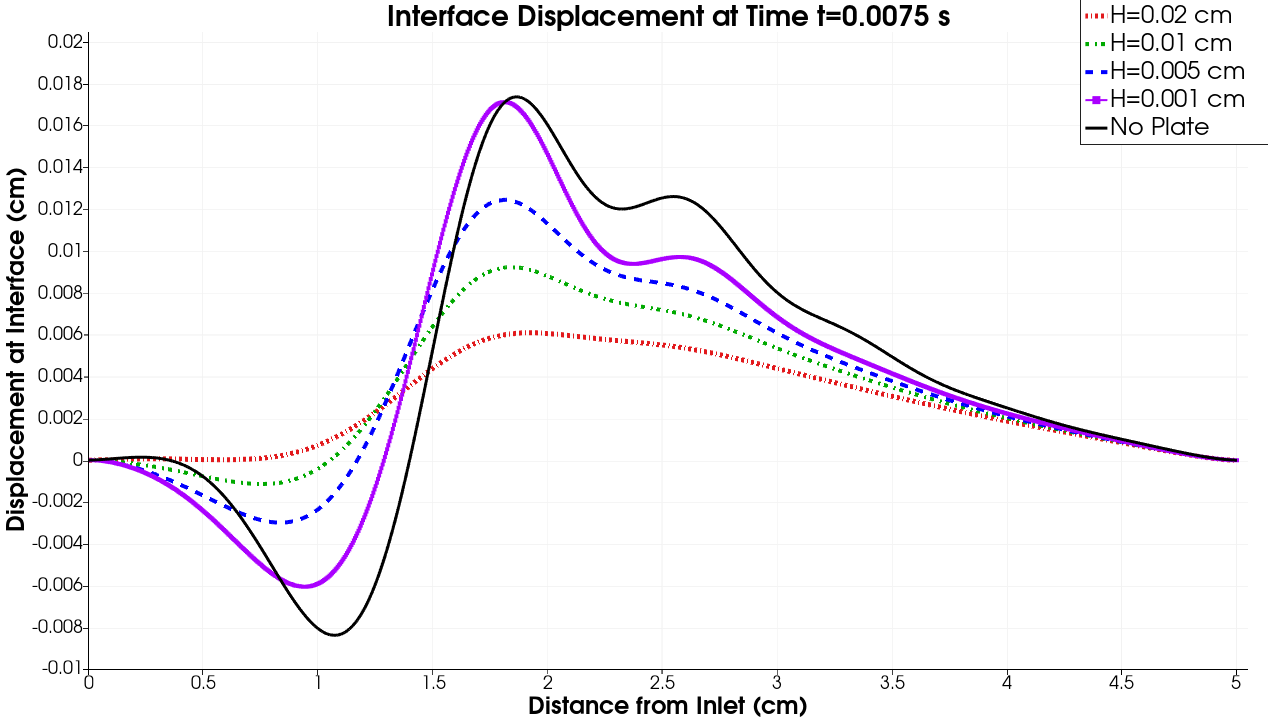}
\caption{Plate/interface displacement after pressure pulse for different values of $H$}\label{fig:Hconv}
\end{figure}

\begin{figure}[h]
\centering
\includegraphics[scale=.3]{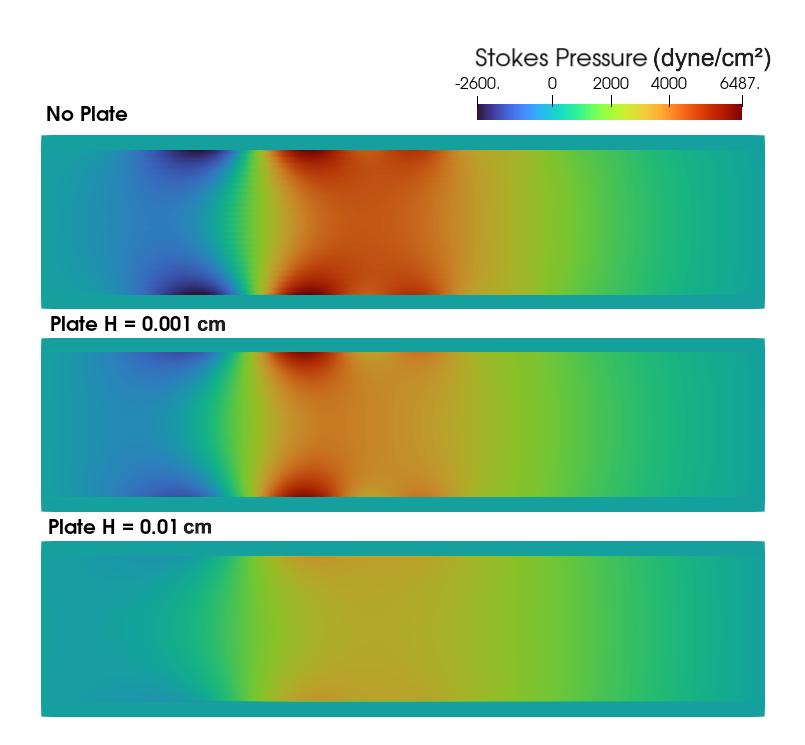}
\caption{Fluid pressure at time $t = .0075$ s for different values of $H$} \label{fig:velFlow}
\end{figure}

\section{Conclusions} 

In this manuscript, we introduced the first numerical scheme for solving a fluid–poroelastic structure interaction (FPSI) problem involving the Stokes flow of an incompressible, viscous fluid coupled with a Biot poroelastic medium across a thin interface that carries mass and poroelastic energy, modeled by the Biot plate equations. The proposed scheme is partitioned in nature, combining the backward Euler Stokes–Biot splitting with the fixed-strain Biot splitting approach. We established the conditional stability of the scheme and validated its accuracy using the method of manufactured solutions. Additionally, we demonstrated convergence of solutions 
as the plate thickness tends to zero, to the solution of the corresponding Stokes–Biot problem without the plate,  providing further validation of the method.

\bibliographystyle{siamplain}
\bibliography{PoroSplitPaper}
\end{document}